\documentclass[12pt, reqno, twoside, letterpaper]{amsart}

\usepackage{amsmath,amssymb,amsbsy,amsfonts,amsthm,latexsym,
	amsopn,amstext,amsxtra,euscript,amscd,stmaryrd,mathrsfs,
	cite,array,tikz}
\usepackage{booktabs}
\usepackage{multirow}

\numberwithin{equation}{section}

\usepackage{color}

\def\eps{\varepsilon}

\def\fl#1{\left\lfloor#1\right\rfloor}

\def\({\left(}
\def\){\right)}

\newcommand{\ind}[1]{\ensuremath{\mathbf{1}_{#1}}}  
\newcommand{\e}{\ensuremath{\mathbf{e}}}

\newcommand{\cA}{\ensuremath{\mathcal{A}}}
\newcommand{\cB}{\ensuremath{\mathcal{B}}}

\newcommand{\cD}{\ensuremath{\mathcal{D}}}

\newcommand{\cF}{\ensuremath{\mathcal{F}}}

\newcommand{\cI}{\ensuremath{\mathcal{I}}}
\newcommand{\cJ}{\ensuremath{\mathcal{J}}}

\newcommand{\cL}{\ensuremath{\mathcal{L}}}

\newcommand{\cN}{\ensuremath{\mathcal{N}}}

\newcommand{\cS}{\ensuremath{\mathcal{S}}}

\newcommand{\cT}{\ensuremath{\mathcal{T}}}

\newcommand{\CC}{\ensuremath{\mathbb{C}}}

\newcommand{\NN}{\ensuremath{\mathbb{N}}}

\newcommand{\RR}{\ensuremath{\mathbb{R}}}

\usepackage[margin=2.5cm]{geometry}

\usepackage{amsthm}
\newtheoremstyle{customthm}
{1em}                    
{1em}                    
{\itshape}               
{}                       
{\scshape}               
{.}                      
{5pt plus 1pt minus 1pt} 
{}                       

\newtheoremstyle{customrem}
{1em}                    
{1em}                    
{}                       
{}                       
{\scshape}               
{.}                      
{5pt plus 1pt minus 1pt} 
{}                       

\theoremstyle{customthm}

\newtheorem{X}{X}[section]
\newtheorem{theorem}[X]{Theorem}  
  
\newtheorem{lemma}[X]{Lemma}

\newtheorem{proposition}[X]{Proposition}

\theoremstyle{customrem}

\usepackage{etoolbox}

\AtEndEnvironment{definition}{\null\hfill\qedsymbol}

\renewcommand{\le}{\ensuremath{\leqslant}}
\renewcommand{\ge}{\ensuremath{\geqslant}}
\def\fl#1{\left\lfloor#1\right\rfloor}

\renewcommand{\pod}[1]{\mathchoice
	{\allowbreak \if@display \mkern 5mu\else \mkern 5mu\fi (#1)}
	{\allowbreak \if@display \mkern 5mu\else \mkern 5mu\fi (#1)}
	{\mkern4mu(#1)}
	{\mkern4mu(#1)}
}

\newcommand*{\defeq}{\mathrel{\vcenter{\baselineskip0.5ex \lineskiplimit0pt
			\hbox{\scriptsize.}\hbox{\scriptsize.}}}%
	=}

\newcommand*{\eqdef}{=\mathrel{\vcenter{\baselineskip0.5ex \lineskiplimit0pt
			\hbox{\scriptsize.}\hbox{\scriptsize.}}}%
}

\DeclareSymbolFont{EUEX}{U}{euex}{m}{n}

\DeclareSymbolFont{euexlargesymbols}{U}{euex}{m}{n}
\DeclareMathSymbol{\intop}{\mathop}{euexlargesymbols}{"52}
\def\int{\intop\nolimits}

\DeclareSymbolFont{euexsymbols}     {U}{euex}{m}{n}
\DeclareMathSymbol{\smallint}{\mathop}{euexsymbols}{"52}

\usepackage{ifthen}
\usepackage{xifthen}

\def\sums{                        
	\@ifnextchar[
	{\sums@i}
	{\ensuremath{\sum}}    
}
\def\sums@i[#1]{
	\@ifnextchar[
	{\sums@ii{#1}}
	{\ensuremath{\sum_{#1}}}
}
\def\sums@ii#1[#2]{
	\@ifnextchar[
	{\sums@iii{#1}{#2}}
	{\ensuremath{\sum_{\substack{#1 \\ #2}}}}
}
\def\sums@iii#1#2[#3]{
	\@ifnextchar[
	{\sums@iv{#1}{#2}{#3}}
	{\ensuremath{\sum_{\substack{#1 \\ #2 \\ #3}}}}
}
\def\sums@iv#1#2#3[#4]{
	\@ifnextchar[
	{\sums@v{#1}{#2}{#3}{#4}}
	{\ensuremath{\sum_{\substack{#1 \\ #2 \\ #3 \\ #4}}}}
}
\def\sums@v#1#2#3#4[#5]{
	{\ensuremath{\sum_{\substack{#1 \\ #2 \\ #3 \\ #4 \\ #5}}}}
}

\def\sumss[#1]{
	\@ifnextchar[
	{\sumss@i[#1]}
	{
		\ifthenelse{\isempty{#1}} 
		{\ensuremath{\sum}}       
		{
			\ifthenelse{\equal{#1}{'}}
			{\ensuremath{\sideset{}{^{\prime}}{\sum}}}
			{\ensuremath{\sideset{}{^{#1}}{\sum}}} 
		}  
	}
}    
\def\sumss@i[#1][#2]{
	\@ifnextchar[
	{\sumss@ii[#1]{#2}}
	{
		\ifthenelse{\isempty{#1}} 
		{\ensuremath{\sum_{#2}}}       
		{
			\ifthenelse{\equal{#1}{'}}
			{\ensuremath{\sideset{}{^{\prime}}{\sum}_{#2}}}
			{\ensuremath{\sideset{}{^{#1}}{\sum}_{#2}}} 
		}  
	}
}
\def\sumss@ii[#1]#2[#3]{
	\@ifnextchar[
	{\sumss@iii[#1]{#2}{#3}}
	{
		\ifthenelse{\isempty{#1}} 
		{\ensuremath{\sum_{\substack{#2 \\ #3}}}}       
		{
			\ifthenelse{\equal{#1}{'}}
			{\ensuremath{\sideset{}{^{\prime}}{\sum}_{\substack{#2 \\ #3}}}}
			{\ensuremath{\sideset{}{^{#1}}{\sum}_{\substack{#2 \\ #3}}}}
		} 
	}
}
\def\sumss@iii[#1]#2#3[#4]{
	\@ifnextchar[
	{\sumss@iv[#1]{#2}{#3}{#4}}
	{
		\ifthenelse{\isempty{#1}} 
		{\ensuremath{\sum_{\substack{#2 \\ #3 \\ #4}}}}       
		{
			\ifthenelse{\equal{#1}{'}}
			{\ensuremath{\sideset{}{^{\prime}}{\sum}_{\substack{#2 \\ #3 \\ #4}}}}
			{\ensuremath{\sideset{}{^{#1}}{\sum}_{\substack{#2 \\ #3 \\ #4}}}} 
		}  
	}
}
\def\sumss@iv[#1]#2#3#4[#5]{
	\@ifnextchar[
	{\sumss@v[#1]{#2}{#3}{#4}{#5}}
	{
		\ifthenelse{\isempty{#1}} 
		{\ensuremath{\sum_{\substack{#2 \\ #3 \\ #4 \\ #5}}}}       
		{
			\ifthenelse{\equal{#1}{'}}
			{\ensuremath{\sideset{}{^{\prime}}{\sum}_{\substack{#2 \\ #3 \\ #4 \\ #5}}}}
			{\ensuremath{\sideset{}{^{#1}}{\sum}_{\substack{#2 \\ #3 \\ #4 \\ #5}}}} 
		}  
	}
}
\def\sumss@v[#1]#2#3#4#5[#6]{
	{\ifthenelse{\isempty{#1}} 
		{\ensuremath{\sum_{\substack{#2 \\ #3 \\ #4 \\ #5 \\ #6 }}}}       
		{
			\ifthenelse{\equal{#1}{'}}
			{\ensuremath{\sideset{}{^{\prime}}{\sum}_{\substack{#2 \\ #3 \\ #4 \\ #5 \\ #6 }}}}
			{\ensuremath{\sideset{}{^{#1}}{\sum}_{\substack{#2 \\ #3 \\ #4 \\ #5 \\ #6 }}}} 
		}  
	}
}

\def\sumsstxt[#1]{
	\@ifnextchar[
	{\sumsstxt@i[#1]}
	{
		\ifthenelse{\isempty{#1}} 
		{\ensuremath{\textstyle\sum}}       
		{
			\ifthenelse{\equal{#1}{'}}
			{\ensuremath{\sideset{}{^{\prime}}{\textstyle\sum}}}
			{\ensuremath{\sideset{}{^{#1}}{\textstyle\sum}}} 
		}  
	}
}    
\def\sumsstxt@i[#1][#2]{
	\@ifnextchar[
	{\sumsstxt@ii[#1]{#2}}
	{
		\ifthenelse{\isempty{#1}} 
		{\ensuremath{\textstyle\sum_{#2}}}       
		{
			\ifthenelse{\equal{#1}{'}}
			{\ensuremath{\sideset{}{^{\prime}}{\textstyle\sum}_{#2}}}
			{\ensuremath{\sideset{}{^{#1}}{\textstyle\sum}_{#2}}} 
		}  
	}
}
\def\sumsstxt@ii[#1]#2[#3]{
	\@ifnextchar[
	{\sumsstxt@iii[#1]{#2}{#3}}
	{
		\ifthenelse{\isempty{#1}} 
		{\ensuremath{\textstyle\sum_{\substack{#2 \\ #3}}}}       
		{
			\ifthenelse{\equal{#1}{'}}
			{\ensuremath{\sideset{}{^{\prime}}{\textstyle\sum}_{\substack{#2 \\ #3}}}}
			{\ensuremath{\sideset{}{^{#1}}{\textstyle\sum}_{\substack{#2 \\ #3}}}}
		} 
	}
}
\def\sumsstxt@iii[#1]#2#3[#4]{
	\@ifnextchar[
	{\sumsstxt@iv[#1]{#2}{#3}{#4}}
	{
		\ifthenelse{\isempty{#1}} 
		{\ensuremath{\textstyle\sum_{\substack{#2 \\ #3 \\ #4}}}}       
		{
			\ifthenelse{\equal{#1}{'}}
			{\ensuremath{\sideset{}{^{\prime}}{\textstyle\sum}_{\substack{#2 \\ #3 \\ #4}}}}
			{\ensuremath{\sideset{}{^{#1}}{\textstyle\sum}_{\substack{#2 \\ #3 \\ #4}}}} 
		}  
	}
}
\def\sumsstxt@iv[#1]#2#3#4[#5]{
	{\ifthenelse{\isempty{#1}} 
		{\ensuremath{\textstyle\sum_{\substack{#2 \\ #3 \\ #4 \\ #5}}}}       
		{
			\ifthenelse{\equal{#1}{'}}
			{\ensuremath{\sideset{}{^{\prime}}{\textstyle\sum}_{\substack{#2 \\ #3 \\ #4 \\ #5}}}}
			{\ensuremath{\sideset{}{^{#1}}{\textstyle\sum}_{\substack{#2 \\ #3 \\ #4 \\ #5}}}} 
		}  
	}
}

\def\prods{              
	\@ifnextchar[
	{\prods@i}
	{\ensuremath{\prod}}    
}
\def\prods@i[#1]{
	\@ifnextchar[
	{\prods@ii{#1}}
	{\ensuremath{\prod_{#1}}}
}
\def\prods@ii#1[#2]{
	\@ifnextchar[
	{\prods@iii{#1}{#2}}
	{\ensuremath{\prod_{\substack{#1 \\ #2}}}}
}
\def\prods@iii#1#2[#3]{
	\@ifnextchar[
	{\prods@iv{#1}{#2}{#3}}
	{\ensuremath{\prod_{\substack{#1 \\ #2 \\ #3}}}}
}
\def\prods@iv#1#2#3[#4]{
	{\ensuremath{\prod_{\substack{#1 \\ #2 \\ #3 \\ #4}}}}
}
\newcommand{\RNum}[1]{\uppercase\expandafter{\romannumeral #1\relax}}
\allowdisplaybreaks

\title[Piatetski-Shapiro Primes in short intervals]
{Piatetski-Shapiro Primes in short intervals}

\author[Lingyu Guo]{Lingyu Guo}

\address{School of Mathematics and Statistics, Xi'an Jiaotong University, Xi'an,Shaanxi,China.}

\email{guo.lingyu@foxmail.com}

\author[Victor Zhenyu Guo]{Victor Zhenyu Guo}

\address{School of Mathematics and Statistics, Xi'an Jiaotong University, Xi'an, Shaanxi, China.}

\email{guozyv@xjtu.edu.cn; vzguo@foxmail.com}

\date{\today}

\begin{document}
	
	\begin{abstract}
    The existence of primes in a short interval, which asks if there are prime numbers in the interval $[x, x + x^\theta]$, is a core problem in number theory. Guth and Maynard proved the best known result for this problem with an asymptotic formula while Baker, Harman and Pintz proved the best lower bound result. 
    
    In this article, we focus on Piatetski-Shapiro primes in a short interval. The study of Piatetski-Shapiro primes of the form $\lfloor n^c \rfloor$ is an approximation of the well-known conjecture that there exist infinitely many primes of the form $n^2+1$. We prove the existence of such primes under restrictions on $\theta$ and $c$ with an asymptotic formula and a lower bound, respectively.  
	\end{abstract}
	
	\maketitle
	
	\begin{quote}
	\end{quote}

	\begin{quote}
	\end{quote}
	
	
	
	\tableofcontents
	
	\newpage
	\section{Introduction}
	
The existence of primes in a short interval, which asks if there is a prime number in the interval $[x, x + x^\theta]$, is a core problem in number theory. Hoheisel \cite{Hoh1930} firstly proved that $\theta = 1 - 1/33000$ is admissible. A remarkable result is by Huxley \cite{Hux}, who proved that 
\begin{equation}
	\label{eq:primeasy}
	\pi(x+x^\theta) - \pi(x)\sim \frac{x^\theta}{\log x},
\end{equation}
where $\pi(x)$ is the prime counting function and $\theta > 7/12$. Guth and Maynard \cite{GM2025} proved the best range of this result with $\theta > 17/30$. If one considers a lower bound result of \eqref{eq:primeasy} instead of an asymptotic formula, Baker, Harman and Pintz \cite{BHP} gave the current best result that for $\theta > 0.525$ and all large enough $x$ it follows that
$$
\pi(x+x^\theta) - \pi(x) \gg \frac{x^\theta}{\log x}. 
$$

The Piatetski-Shapiro sequences are sequences of the form
$$
\cN_c \defeq (\lfloor n^{c} \rfloor)_{n=1}^\infty,
$$
where $\fl{\cdot}$ is the integer part. Piatetski-Shapiro~\cite{PS} proved the Piatetski-Shapiro prime number theorem stating that for $ 1 < c < 12/11$ the counting function
$$
\pi_c(x) \defeq \# \big\{\text{\rm prime~} p\le x : p \in \cN_{c} \big\}
$$
satisfies the asymptotic relation
\begin{equation*}
	\pi_c(x) = (1 + o(1)) \frac{x^{1/c}}{\log x} \qquad \text{~\rm as } x \to \infty.
\end{equation*}
The admissible range for $c$ of the above formula has been extended many times and is currently known to hold for all $1 < c < 2817/2426$ thanks to Rivat and Sargos~\cite{RiSa}. Rivat and Wu~\cite{RiWu} also showed that there are infinitely many Piatetski-Shapiro primes for $1 < c < 243/205$ without an asymptotic formula. The estimation of Piatetski-Shapiro primes is an approximation of the well-known conjecture that there exist infinitely many primes of the form $n^2+1$. We refer the readers to \cite{GGL} for a survey on the investigations of admissible ranges for $c$ in this problem. 

We are interested in counting Piatetski-Shapiro primes in a short interval. For a result with an asymptotic formula, we prove the following theorem; also see Figure \ref{fig:thm1} about the admissible range. 

\begin{theorem}\label{thm1}
	Let $\theta \in (2/3,1)$ and $\gamma \defeq c^{-1}<1$. We have
\begin{align*}
\pi_{c}(x+x^{\theta}) - \pi_{c}(x) = \frac{\gamma x^{\theta+\gamma-1}}{\log x} + O\Big(\frac{x^{\theta+\gamma-1}}{\log^2 x}\Big)
\end{align*}
provided that 

\begin{equation}\label{eq:thm1}
\gamma > 
\begin{cases}
	\dfrac{20-3\theta}{18}, & \text{if} \quad \dfrac{2}{3} < \theta \le \dfrac{40}{51}; \\[1em]
	\max \left( \dfrac{1+\theta}{2},\; \dfrac{10-4\theta}{7},\; \dfrac{17-3\theta}{15} \right), & \text{if} \quad \dfrac{40}{51} < \theta \le \dfrac{10}{11}; \\[1em]
	\max \left( \dfrac{15-6\theta}{10},\; \dfrac{17-3\theta}{15} \right), & \text{if} \quad \dfrac{10}{11} < \theta < 1.
\end{cases}
\end{equation}

\end{theorem}	

\begin{figure}[htbp]
	\centering
	\includegraphics[width=0.92\textwidth]{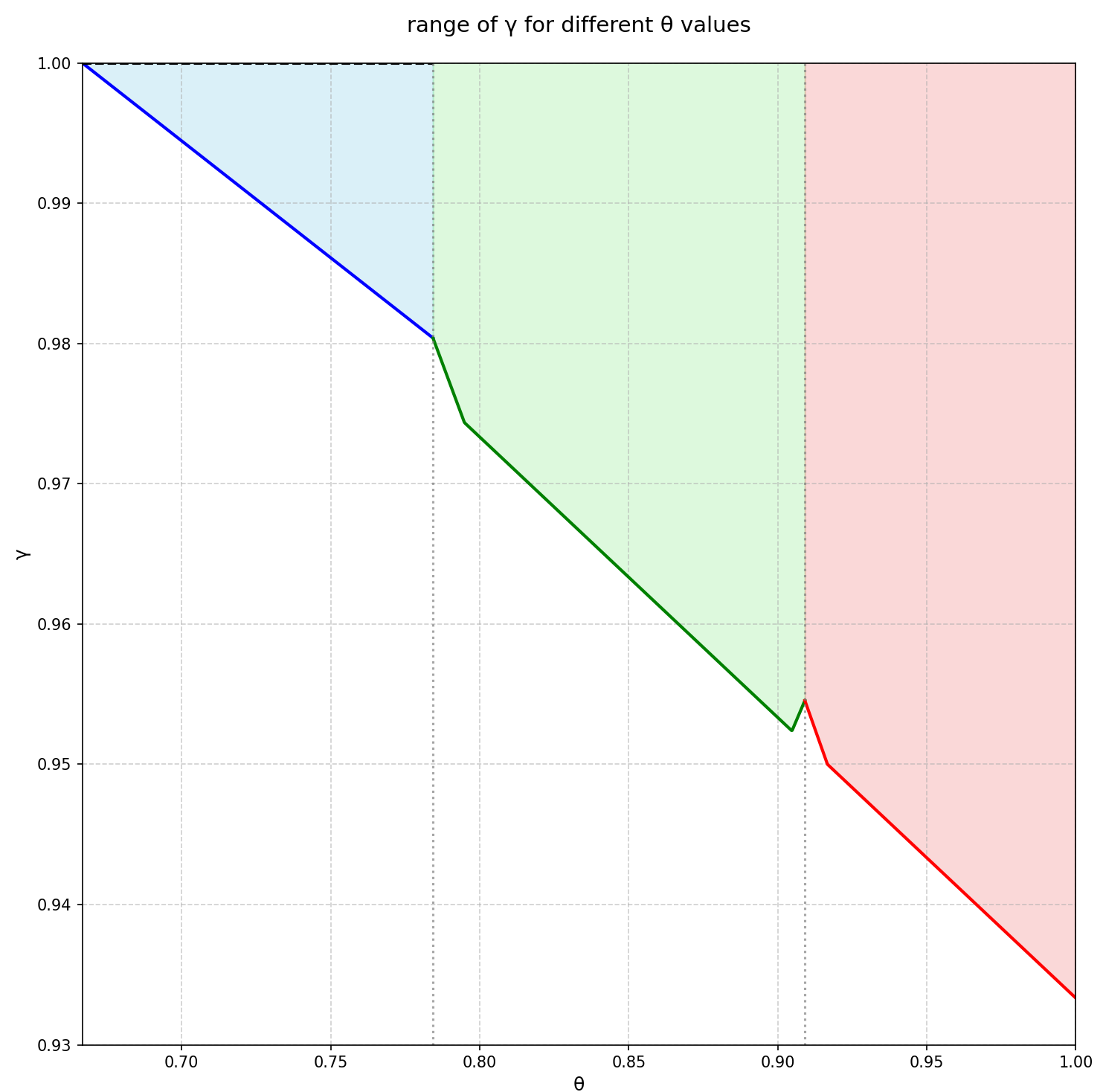}
	\caption{The admissible range of $\gamma$ in Theorem~\ref{thm1}}\label{fig:thm1}
\end{figure}

It is unusual that $\gamma$ increases with $\theta$ near $0.9$ in Figure \ref{fig:thm1}. This phenomenon is due to the term $(1+\theta)/2$ in Proposition~\ref{lem:type2_2}. Actually, we are able to prove an admissible type II sum when $\gamma < (1+\theta)/2$. However, we ignore this case since we focus on the situation when $\theta$ is small. This case appears only when $\theta > 0.8$, so to make our proposition brief we decide not to spend more effort on this more technical consideration.  

If we consider a lower bound instead of an asymptotic formula, the method of Harman sieve can be applied, so the range of $c$ is improved as the following theorem. 

\begin{theorem}\label{thm2}
For some fixed values of $\theta \in (2/3,1)$ and $\gamma>\gamma_0$ in Table~\ref{table of thm2} , we have
	\begin{align*}
		\pi_{c}(x+x^{\theta}) - \pi_{c}(x) \gg \frac{ x^{\theta+\gamma-1}}{\log x}.
	\end{align*}
\begin{table}[htbp]	
		\centering
		\caption{Ranges of $\gamma_0$ for fixed $\theta$}
		\begin{tabular}{cccccc}
			\toprule
			$\theta$ & $\gamma_0$ & $\theta$ & $\gamma_0$ & $\theta$ & $\gamma_0$\\
			\midrule
$0.67$    & $0.9975$    & $0.78$    & $0.9254$    & $0.89$    & $0.9235$ \\
$0.68$    & $0.9900$    & $0.79$    & $0.9221$    & $0.90$    & $0.9168$ \\
$0.69$    & $0.9825$    & $0.80$    & $0.9188$    & $0.91$    & $0.9101$ \\
$0.70$    & $0.9750$    & $0.81$    & $0.9155$    & $0.92$    & $0.9034$ \\
$0.71$    & $0.9675$    & $0.82$    & $0.9122$    & $0.93$    & $0.8967$ \\
$0.72$    & $0.9600$    & $0.83$    & $0.9150$    & $0.94$    & $0.8900$ \\
$0.73$    & $0.9525$    & $0.84$    & $0.9200$    & $0.95$    & $0.8833$ \\
$0.74$    & $0.9450$    & $0.85$    & $0.9250$    & $0.96$    & $0.8766$ \\
$0.75$    & $0.9375$    & $0.86$    & $0.9300$    & $0.97$    & $0.8699$ \\
$0.76$    & $0.9320$    & $0.87$    & $0.9350$    & $0.98$    & $0.8633$ \\
$0.77$    & $0.9287$    & $0.88$    & $0.9302$    & $0.99$    & $0.8566$ \\
			\bottomrule
		\end{tabular}
		\label{table of thm2}
	\end{table}
\end{theorem}	

Note that Theorem~\ref{thm2} gives a better range of $\gamma$ than Theorem~\ref{thm1} for any $2/3 < \theta < 1$, not only the values in Table~\ref{table of thm2}. The reason why we do not give a relation between $\theta$ and $\gamma$ as Theorem~\ref{thm1} is that the integral of \eqref{s6tos12} is hard to calculate if we do not fix $\theta$ and $\gamma$ when we apply Harman sieve. Check Table \ref{table of thm1 and thm2} for a comparison of admissible ranges for $\gamma$ with corresponding $\theta$ between Theorem \ref{thm1} and Theorem \ref{thm2}. We remark that it is possible to imcrease the range of $\gamma$ when $\gamma$ is big ($>0.87$) by a more careful calculation of exponential sums or Harman sieves, since it is not the main purpose of this article. 

\begin{table}[htbp]	
	
	\centering
	\caption{Ranges of $\gamma$ and $\theta$ for Theorem~\ref{thm1} and Theorem~\ref{thm2} }
	\begin{tabular}{cccccc}
		\toprule
		$\theta$ & $\gamma$ in Thm~\ref{thm1} & $\gamma$ in Thm~\ref{thm2} & $\theta$ & $\gamma$ in Thm~\ref{thm1} & $\gamma$ in Thm~\ref{thm2} \\
		\midrule
		$0.67$   & $0.9995$   & $0.9975$  & $0.84$ & $0.9654$ & $0.9200$\\
		$0.68$   & $0.9978$   & $0.9900$  & $0.85$ & $0.9634$ & $0.9250$\\
		$0.69$   & $0.9962$   & $0.9825$  & $0.86$ & $0.9614$ & $0.9300$\\
		$0.70$   & $0.9945$   & $0.9750$  & $0.87$ & $0.9594$ & $0.9350$\\
		$0.71$   & $0.9928$   & $0.9675$  & $0.88$ & $0.9574$ & $0.9302$\\
		$0.72$   & $0.9912$   & $0.9600$  & $0.89$ & $0.9554$ & $0.9235$\\
		$0.73$   & $0.9895$   & $0.9525$  & $0.90$ & $0.9534$ & $0.9168$\\
		$0.74$   & $0.9878$   & $0.9450$  & $0.91$ & $0.9540$ & $0.9101$\\
		$0.75$   & $0.9862$   & $0.9375$  & $0.92$ & $0.9494$ & $0.9034$\\
		$0.76$   & $0.9845$   & $0.9320$   & $0.93$ & $0.9474$ & $0.8967$\\
		$0.77$   & $0.9828$   & $0.9287$   & $0.94$ & $0.9454$ & $0.8900$\\
		$0.78$   & $0.9812$   & $0.9254$   & $0.95$ & $0.9434$ & $0.8833$\\
		$0.79$   & $0.9772$   & $0.9221$   & $0.96$ & $0.9414$ & $0.8766$\\
		$0.80$   & $0.9734$   & $0.9188$   & $0.97$ & $0.9394$ & $0.8699$\\
		$0.81$   & $0.9714$   & $0.9155$   & $0.98$ & $0.9374$ & $0.8633$\\
		$0.82$   & $0.9694$   & $0.9122$   & $0.99$ & $0.9354$ & $0.8566$\\
		$0.83$   & $0.9674$   & $0.9150$   &        &          &         \\
		\bottomrule
	\end{tabular}
	\label{table of thm1 and thm2}
\end{table}

A key part is the estimation of the following exponential sum
\begin{equation}
\label{eq:exp}
\sum_{h \sim H} \delta_h \sum_{n \sim N}\sum_{\substack{m \\ mn \in I}}  a_n b_m \e(h(mn)^\gamma),
\end{equation}
where $|\delta_h| \leqslant 1$ and $H \leqslant x^{1-\gamma+\epsilon}$. The interval $I = (x, x+y]$ for $y \defeq x^\theta$. If the coefficients $a_n$ and $b_m$ satisfy the conditions
$$|
a_n| \ll x^{\varepsilon}, \quad b_m = 1 \quad \text{or} \quad b_m = \log m,
$$
the sum is a type I sum and denoted by $S_I$; if they satisfy the conditions
$$|a_n| \ll x^{\varepsilon}, \quad |b_m| \ll x^{\varepsilon},$$
the sum is a type II sum and denoted by $S_{II}$ in this paper.

We mention that the exponential sum of the form
\begin{equation}
\label{eq:expn}
\sum_{h \sim H} \delta_h \sum_{n \sim N}\sum_{m \sim M}  a_n b_m \e(h(mn)^\gamma)
\end{equation}
was well investigated during the research of Piatetski-Shapiro primes; see Robert and Sargos \cite{RS2006} for a type I sum and Heath-Brown \cite{HB2} for a type II sum. However, to bound the exponential sum as \eqref{eq:exp}, we need to employ the short interval information $mn \in I$, so methods for \eqref{eq:expn} may not provide the best bound here. 
	
For type I sums (Proposition \ref{lem:type1}), a normal treatment is to consider the cancellation of $h,m,n$ via a double large sieve technique. However, we start by estimating the inner sum of $m$ by Lemma~\ref{lem:high-drivative}, since the condition of the short interval entangles the relationship between $m$ and $n$, which makes it hard to use other methods. Another reason is that in this question we try to make $\theta$ as small as possible instead of enlarging the range of $\gamma$ which leads that other normal methods do not work well. We mention that if $\theta>0.87$, combining double large sieve techique and our methods may work better, but as mentioned before, we do not focus on this case.  

For type II sums, we apply two different methods, which lead to Proposition \ref{lem:type2_2} and Proposition \ref{lem:type2_1}. The first method of type II sum is by switching it into an estimation of a type I sum via the Cauchy-Schwarz inequality and A-process, which fits the best to short interval information. To grab more type II information for sieves, we also apply Heath-Brown's idea in \cite{HB2} with counting rational points as a second method. 

We prove the upper bounds of exponential sums in Section \ref{sec:Bounds on exponential sums}. By the bounds, we provide an asymptotic formula for Piatetski-Shapiro primes in short intervals in Section \ref{sec:An asymptotic result}. The lower bound result (Theorem \ref{thm2}) and the set-up of Harman sieve are described in Section \ref{sec:lower bound}. 

\section{Preliminaries}
	
\subsection{Notations}
	
	We denote by $\fl{t}$ and $\{t\}$ the integer part and the fractional part of $t$, respectively. As is customary, we put $\e(t)\defeq e^{2\pi it}$. We make considerable use of the sawtooth function defined by
	$$
	\psi(t) \defeq t-\fl{t}-\frac{1}{2}=\{t\}-\frac{1}{2}\qquad(t\in\RR).
	$$
	The letter $p$ always denotes a prime. The interval $I$ always denotes $(x,x+y]$ and $y$ is defined as $x^{\theta}$ where $1/2<\theta \leqslant 1$. For the Piatetski-Shapiro sequence $(\fl{n^c})_{n=1}^\infty$, we denote $\gamma \defeq c^{-1}$. 
	We use notation of the form $m\sim M$ as an abbreviation for $M< m\le 2M$. $\eps$ is always a sufficiently small positive number. For an arbitrary set $\cS$, we use $\ind{\cS}$ to denote its indicator function:
	$$
	\ind{\cS}(n)\defeq\begin{cases}
		1&\quad\hbox{if $n\in\cS$,}\\
		0&\quad\hbox{if $n\not\in\cS$.}\\
	\end{cases}
	$$
	
	Throughout the paper, implied constants in symbols $O$, $\ll$ and $\gg$ may depend (where obvious) on the parameters $c, \eps$ but are absolute otherwise. For given functions $F$ and $G$, the notations $F\ll G$, $G\gg F$ and $F=O(G)$ are all equivalent to the statement that the inequality $|F|\le C|G|$ holds with some constant $C>0$. $F \asymp G$ means that $F \ll G \ll F$. 
	
\subsection{Technical lemmas}
	
	\begin{lemma}
		\label{lem:Vaaler}
		For any $H\ge 1$, there exist numbers $a_h,b_h$ such that
		$$
		\bigg|\psi(t)-\sum_{0<|h|\le H}a_h\,\e(th)\bigg|
		\le\sum_{|h|\le H}b_h\,\e(th),\quad
		a_h\ll 1/|h|,\quad b_h\ll 1/H.
		$$
	\end{lemma}
	
	\begin{proof}
		See \cite{Vaal} by Vaaler.
	\end{proof}
	
	\begin{lemma}
		\label{lem:PS}
		A natural number $m$ has the form $\fl{n^c}$ if and only if $\mathbf{1}_{\cN_c}(m) = 1$, where
		$\mathbf{1}_{\cN_c}(m) \defeq \fl{-m^\gamma} - \fl{-(m+1)^\gamma}$.  Moreover,
		$$
		\mathbf{1}_{\cN_c}(m)=\gamma m^{\gamma-1}+ \psi(-(m+1)^\gamma) - \psi(-m^\gamma)+O(m^{\gamma-2}).
		$$
	\end{lemma}
	
	\begin{proof}
		The equality $m = \lfloor n^c \rfloor$ holds precisely when $m \leqslant n^c < m + 1$, or equivalently, when $-(m + 1)^\gamma \leqslant -n < -m^\gamma$. Consequently,
		\begin{align*}
			\mathbf{1}_{\cN_c}(m) &=\lfloor -m^\gamma \rfloor - \lfloor -(m + 1)^\gamma \rfloor=(m + 1)^\gamma - m^\gamma +\psi(-(m + 1)^\gamma) - \psi(-m^\gamma) \\
			&=\gamma m^{\gamma-1}+ \psi(-(m+1)^\gamma) - \psi(-m^\gamma)+O(m^{\gamma-2}).
		\end{align*}
	\end{proof}
The following lemma is the famous Weyl-van der Corput inequality, also called the A-process; see \cite[Lemma 2.5]{GraKol}. 
    \begin{lemma}
		\label{lem:A}
		Suppose $f(n)$ is a complex valued function and $I$ is an interval such that $f(n) = 0$ if $n \notin I$. If $H$ is a positive integer then
		$$
		|\sum_{n \in I} f(n)|^2 \le \frac{|I| + H}{H} \sum_{|h|< H} \big(1 - \frac{h}{H} \big) \sum_{n \in I} f(n) \overline{f(n-h)}.  
		$$
	\end{lemma}
	
	\begin{lemma}\label{lem:2-drivative}
Suppose that $ f $ is a real valued function with two continuous derivatives on $ I $. Suppose also that there is some $ \lambda > 0 $ and some $ \alpha \geqslant 1 $ such that  
$$
\lambda \leqslant |f''(x)| \leqslant \alpha\lambda
$$ 
on $ I $. Then  
$$
\sum_{n \in I} e(f(n)) \ll \alpha |I|\lambda^{\frac{1}{2}} + \lambda^{-\frac{1}{2}}.
$$
	\end{lemma}
	
	\begin{proof}
		See \cite[Theorem~2.2]{GraKol}.
	\end{proof}
	
	\begin{lemma}\label{lem:high-drivative}
Let $ q $ be a positive integer. Suppose that $ f $ is a real valued function with $ q + 2 $ continuous derivatives on $ I $. Suppose also that for some $ \lambda > 0 $ and for some $ \alpha \geqslant 1 $,
$$
\lambda \leqslant \left| f^{(q+2)}(x) \right| \leqslant \alpha\lambda
$$
on $ I $. Let $ Q \defeq 2^q $. Then
$$
\sum_{n \in I} e(f(n)) \ll |I|(\alpha^2\lambda)^{1/(4Q-2)} + |I|^{1-1/2Q}\alpha^{1/2Q} + |I|^{1-2/Q+1/Q^2}\lambda^{-1/2Q}.
$$
The implied constant is absolute.
	\end{lemma}
	
	\begin{proof}
		See \cite[Theorem~2.8]{GraKol}.
	\end{proof}
	
	\begin{lemma} \label{lemmea: balance} Let
		$$
		L(E)=\sum_{1\leqslant i\leqslant u}A_iE^{a_i}+\sum_{1\leqslant j\leqslant v}B_jE^{-b_j},
		$$
		where $A_i,B_j,a_i$ and $b_j$ are positive. Let $0\leqslant E_1\leqslant E_2$. Then there is
		some $E\in(E_1,E_2]$ such that
		$$
		L(E)\ll \sum_{1\leqslant i\leqslant u}\sum_{1\leqslant j\leqslant v}\left(A_i^{b_j}B_j^{a_i}\right)^{1/(a_i+b_j)}
		+\sum_{1\leqslant i\leqslant u}A_iE_1^{a_i}+\sum_{1\leqslant j\leqslant v}B_jE_2^{-b_j},
		$$
		where the implied constant only depends on $u$ and $v$.
	\end{lemma}
	
	\begin{proof}
		See \cite[Lemma~2.4]{GraKol}.
	\end{proof}
	
	\begin{lemma}\label{lemma:FoIw}
		Let $ \alpha \beta \neq 0, \, \Delta > 0, \, M \geqslant 1 $ and $ N \geqslant 1 $. 
		Let $ \mathcal{A}(M, N; \Delta) $ be the number of quadruples $(m, \tilde{m}, n, \tilde{n})$ such that
		$$
		\left| \left( \frac{\tilde{m}}{m} \right)^\alpha - \left( \frac{\tilde{n}}{n} \right)^\beta \right| < \Delta,
	    $$
		with $ M \leqslant m, \, \tilde{m} < 2M $ and $ N \leqslant n, \, \tilde{n} < 2N $. 
		We then have
		$$
		\mathcal{A}(M, N; \Delta) \leqslant MN \log 2MN + \Delta M^2 N^2.
		$$
	\end{lemma}
	
	\begin{proof}
		See \cite[Lemma~1]{FoIw}.
	\end{proof}
	
	\begin{lemma}\label{lemma:perron}
		Suppose $\alpha, \beta > 0, \alpha \neq \beta, T \geqslant 1$. Then		
		$$
		\frac{1}{\pi} \int_{-T}^{T} e^{i\alpha t} \frac{\sin \beta t}{t} \, dt = \delta + O\bigl(T^{-1}|\beta - \alpha|^{-1}\bigr),
		$$		
		where
		$$
		\delta =
		\begin{cases} 
			0 & \text{if } \alpha > \beta, \\
			1 & \text{if } \alpha < \beta.
		\end{cases}
		$$
	\end{lemma}
	
	\begin{proof}
		See \cite[Lemma~2.2]{Ha}.
	\end{proof}

\section{Bounds on exponential sums}\label{sec:Bounds on exponential sums}
Recall that the corresponding type I sum of the form \eqref{eq:exp} is 
\begin{equation}
\label{eq:tyI}
S_I \defeq \sum_{h \sim H} \delta_h  \sum_{n \sim N} a_n \sum_{ \substack{m \\ mn \in I}}  \e(h(mn)^\gamma)
\end{equation}	
and 
the corresponding type II sum is
\begin{equation}
\label{eq:tyII}
S_{II} \defeq \sum_{h \sim H} \delta_h \sum_{n \sim N}\sum_{ \substack{m \\ mn \in I}} a_n b_m \e(h(mn)^\gamma)
\end{equation}
where 
$$
|\delta_h| \ll 1, |a_n|, |b_m| \ll x^\eps. 
$$
We prove the following bounds for type I and type II sums. 

\begin{proposition}[Type I]\label{lem:type1}
Let $\theta \in (2/3,1]$ and $N$ satisfies the condition
\begin{align*}
N \ll x^{2\gamma -\frac{4}{3}-\varepsilon} ,
\end{align*}
and
\begin{align}\label{condition of type1}
\frac{8-3\theta}{6} < \gamma <1,
\end{align}
then
$$
S_I \ll x^{\theta-\varepsilon}.
$$
\end{proposition}

\begin{proposition}[Type II]\label{lem:type2_2}
	Let $\theta \in (2/3,1]$ and $N$ satisfies the condition
	\begin{align*}
		x^{3-\theta-2\gamma+\varepsilon} \ll N \ll x^{\theta + 4\gamma -4-\varepsilon} \quad \text{or} \quad x^{5-\theta-4\gamma +\varepsilon}\ll N \ll x^{\theta+2\gamma-2-\varepsilon},
	\end{align*}
	and
	\begin{align}\label{condition of type2_2}
		\max\Big(\frac{1+\theta}{2} , \frac{6-3\theta}{4}\Big) < \gamma <1,
	\end{align}
	then
	$$
	S_{II} \ll x^{\theta-\varepsilon}.
	$$
\end{proposition}

\begin{proposition}[Type II]\label{lem:type2_1}
	Let $\theta \in (3/4,1]$ and $N$ satisfies the condition
	\begin{align*}
		x^{2-\theta-\gamma +\varepsilon}\ll N \ll x^{3\theta+5\gamma-7-\varepsilon} \quad \text{or} \quad x^{8-3\theta-5\gamma+\varepsilon} \ll N \ll x^{\theta + \gamma -1-\varepsilon} ,
	\end{align*}
	and
	\begin{align}\label{condition of type2_1}
		\frac{9-4\theta}{6} < \gamma <1,
	\end{align}
	then
	$$
	S_{II} \ll x^{\theta-\varepsilon}.
	$$
\end{proposition}

\subsection{Estimate of Type I Sum: Proof of Proposition \ref{lem:type1}}\label{Estimate of Type I Sum}
In this part, we shall bound the type I sum defined as \eqref{eq:tyI}. Unlike an usual approach, we start by estimating the inner sum of $m$ by Lemma~\ref{lem:high-drivative} for the sake of restriction that $mn \in I$, which gives that
$$
x^{-\varepsilon}S_I \ll \sum_{h \sim H} \sum_{n \sim N} \Big| \sum_{ \substack{m \\ mn \in I}}  \e(h(mn)^\gamma) \Big|.
$$
Let $f(m) \defeq h(mn)^{\gamma}$. The $(k+2)$-th derivative is 
$$
f^{(k+2)}(m)=\gamma \cdots (\gamma-k-1)hm^{\gamma-k-2}n^{\gamma} \asymp h \big( \frac{x}{n} \big)^{\gamma-k-2} n^{\gamma} \eqdef \lambda.
$$
By Lemma~\ref{lem:high-drivative}, we derive that
$$
x^{-\varepsilon}S_I \ll \sum_{h \sim H} \sum_{n \sim N} \bigg( \frac{y}{n} \lambda^{1/(4K-2)} + \big( \frac{y}{n} \big)^{1-1/2K} + \big( \frac{y}{n} \big)^{1-2/K+1/K^2} \lambda^{-1/2K} \bigg), 
$$
where $K=2^k$. Sum over $h$ and $n$ to obtain that
\begin{align*}
\nonumber
x^{-\varepsilon}S_I \ll &x^{\theta  - \frac{4K-1}{4K-2}\gamma + \frac{\gamma + 4K-k-3}{4K-2}}N^{\frac{k+2}{4K-2}} + x^{1+\theta - \gamma -\frac{\theta}{2K}}N^{\frac{1}{2k}} \\
&+ x^{1+\theta - \gamma + \frac{k+1 - 4\theta}{2K} + \frac{\theta}{K^2}}N^{\frac{1}{K}-\frac{k}{2K}-\frac{1}{K^2}}.
\end{align*}
Let the upper bound of $S_I \ll x^{\theta-\varepsilon}$ and $k=1$. We achieve that Type I sum holds if 
\begin{align*}
N \ll x^{2\gamma -\frac{4}{3}} 
\end{align*}
with 
$$
\gamma > \frac{8-3\theta}{6} \quad \text{and} \quad \gamma > \frac{6-3\theta}{4}.
$$
Noting that for $\theta \in (2/3,1]$
$$
\frac{8-3\theta}{6} > \frac{6-3\theta}{4}
$$
all the time, we finish the proof of Proposition~\ref{lem:type1}.

\subsection{Estimate of Type II Sum: Proof of Proposition \ref{lem:type2_2}}\label{Estimate of Type II Sum}
In this part, we estimate type II sum defined as \eqref{eq:tyII}. Applying the Cauchy-Schwarz inequality, it follows that
\begin{align*}
	x^{-\varepsilon}S_{II}^2 \ll HN \sum_{h \sim H} \sum_{n \sim N} \Big| \sum_{\substack{m \\ mn\in I}} b_m \e(h(mn)^{\gamma}) \Big|^2.
\end{align*}
By the A-process, namely, Lemma~\ref{lem:A}, we have
\begin{align*}
	x^{-\varepsilon}S_{II}^2 &\ll HN \sum_{h \sim H} \sum_{n \sim N} \frac{y/n+Q}{Q} \sum_{0 \leqslant |q| \leqslant Q}\big( 1-\frac{|q|}{Q} \big) \sum_{\substack{m \\ mn \in I}}b_{m+q}\overline{b_m} \e(hn^\gamma((m+q)^{\gamma}-m^{\gamma})) \\
	& \ll \frac{H^2y^2}{Q} + \frac{Hy}{Q} \sum_{h \sim H} \sum_{1 \leqslant |q| \leqslant Q} \sum_{m \asymp x/N} \Big| \sum_{\substack{n \\ mn\in I}} \e(hn^\gamma((m+q)^{\gamma}-m^{\gamma})) \Big|
\end{align*}
provided $1 \leqslant Q \leqslant y/(2N).$ Applying Lemma~\ref{lem:2-drivative} to the inner sum of variable $n$ we obtain that
\begin{align*}
	x^{-\varepsilon}S_{II}^2 \ll \frac{H^2y^2}{Q} + \frac{Hy}{Q} \sum_{h \sim H} \sum_{1 \leqslant |q| \leqslant Q} \sum_{m \asymp x/N} \bigg( yh^{\frac{1}{2}}q^{\frac{1}{2}}m^{\frac{\gamma-3}{2}}N^{\frac{\gamma}{2}-1} + h^{-\frac{1}{2}}q^{-\frac{1}{2}}m^{\frac{-\gamma+1}{2}}N^{-\frac{\gamma}{2}+1} \bigg).
\end{align*}
Recalling that $H \leqslant x^{1-\gamma+\varepsilon}$ and $y=x^{\theta}$, we have
\begin{align*}
	x^{-\varepsilon}S_{II}^2 \ll Q^{-1}x^{2\theta - 2\gamma +2} + Q^{\frac{1}{2}} x^{2\theta - 2\gamma +2} N^{-\frac{1}{2}} + Q^{-\frac{1}{2}} x^{\theta - 2\gamma +3} N^{-\frac{1}{2}}.
\end{align*}
Optimizing $Q$ over $[1,2x^{\theta}N^{-1}]$ by Lemma \ref{lemmea: balance} we find that when 
$$
\gamma > \frac{1+\theta}{2},
$$
type II sum holds if
\begin{align*}
	x^{5-\theta-4\gamma+\varepsilon} \ll N \ll x^{\theta +2\gamma-2-\varepsilon} 
\end{align*}
with
$$
\gamma > \frac{7-2\theta}{6} \quad \text{and} \quad \gamma > \frac{6-3\theta}{4}.
$$
Noting that for $\theta\in (2/3,1]$
$$
\max\big(\frac{6-3\theta}{4},\frac{1+\theta}{2}\big) \geqslant \frac{7-2\theta}{6}
$$
all the time, we finish the proof of Proposition~\ref{lem:type2_2}.

\subsection{Estimate of Type II Sum: Proof of Proposition \ref{lem:type2_1}}\label{Estimate of Type II' Sum}

Now we try an alternative method due to Heath-Brown \cite{HB2}. 

We give the estimate when $N$ is small first. Since the condition of the short interval entangles the relationship between $m$ and $n$, we cannot split $m$ and $n$ at the same time. Noting that
$$
0<hn^{\gamma} \leqslant 4HN^{\gamma},
$$
we decompose the collection of available pairs $(n, h)$ into sets $\cL_q(1 \leqslant q \leqslant Q)$, defined by
$$
\cL_q \defeq \{(n,h): 4HN^{\gamma}(q-1) < Qhn^{\gamma} \leqslant 4HN^{\gamma}q \}. 
$$
Then
$$
S_{II} = \sum_{1 \leqslant q \leqslant Q} \sum_{ \substack{h \sim H \ n \sim N \\ (n,h) \in \cL_q}} \sum_{mn \in I} \delta_ha_n b_m \e(h(mn)^\gamma).
$$
By the Cauchy-Schwarz inequality, it follows that
\begin{align}
	\label{4.1}
\nonumber
x^{-\varepsilon}|S_{II}|^2 &\ll Q\frac{x}{N} \sum_{1 \leqslant q \leqslant Q} \sum_{m \asymp x/N} \Big| \sum_{\substack{h \sim H \ mn \in I \\ hn^{\gamma} \in \cL_q}} \delta_h a_n \e(h(mn)^{\gamma}) \Big|^2 \\
\nonumber
&\ll Q\frac{x}{N} \sum_{1 \leqslant q \leqslant Q} \sum_{m \asymp x/N} \underset{\substack{(n_1,h_1) \in \cL_q \\ (n_2,h_2) \in \cL_q \\ |n_1-n_2| \ll Nx^{\theta-1}}}{\sum_{\substack{h_1 \sim H \\ h_2 \sim H}} \sum_{\substack{mn_1 \in I \\ mn_2 \in I}}}  \delta_{h_1} \overline{\delta_{h_2}} a_{n_1} \overline{a_{n_2}} \e((h_1n^{\gamma}_1 - h_2n^{\gamma}_2)m^\gamma) \\
&\ll Q\frac{x}{N} \sum_{1 \leqslant q \leqslant Q} \underset{\substack{(n_1,h_1) \in \cL_q \\ (n_2,h_2) \in \cL_q \\ |n_1-n_2| \ll Nx^{\theta-1}}}{\sum_{\substack{h_1 \sim H \\ h_2 \sim H}} \sum_{\substack{n_1 \sim N \\ n_2 \sim N}}} \Big| \sum_{\substack{mn_1\in I \\ mn_2 \in I}} \e((h_1n^{\gamma}_1 - h_2n^{\gamma}_2)m^\gamma) \Big|.
\end{align}
Next we estimate the inner sum of \eqref{4.1} by Lemma~\ref{lem:2-drivative} and the trivial bound. Let
$$
f(m) \defeq (h_1n^{\gamma}_1 - h_2n^{\gamma}_2)m^\gamma ,
$$
we have
$$
f^{''}(m)=\gamma(\gamma-1)(h_1n^{\gamma}_1 - h_2n^{\gamma}_2)m^{\gamma-2} \asymp (h_1n^{\gamma}_1 - h_2n^{\gamma}_2) \big( \frac{x}{N} \big)^{\gamma-2}.
$$
Define $\Delta \defeq h_1n^{\gamma}_1 - h_2n^{\gamma}_2$. We have
\begin{align}\label{4.2}
x^{-\varepsilon}|S_{II}|^2 &\ll Q\frac{x}{N} \sum_{1 \leqslant q \leqslant Q} \underset{\substack{(n_1,h_1) \in \cL_q \\ (n_2,h_2) \in \cL_q \\ |n_1-n_2| \ll Nx^{\theta-1}}}{\sum_{\substack{h_1 \sim H \\ h_2 \sim H}} \sum_{\substack{n_1 \sim N \\ n_2 \sim N}}} \min{\bigg(\frac{y}{N},\frac{y}{N}\Big(\Delta\big( \frac{x}{N} \big)^{\gamma-2}\Big)^{1/2}+\Big(\Delta\big( \frac{x}{N} \big)^{\gamma-2}\Big)^{-1/2}\bigg)}.
\end{align}
One may ask why we do not use other methods here. The answer is similar as before since we only consider the situation when $\theta$ tends to small instead of $\gamma$. The reason why we do not use the method of exponent pairs is that the conditions of the exponent pair suppose the variable lies in an interval $(M,2M]$ such that this method would not use the information of short interval well. One may still try Lemma~\ref{lem:high-drivative}, however the second derivative test is possibly the best in the view of $\theta$.	

Now we continue to estimate $S_{II}$. Firstly, summing over $q$, the contribution of the term 
$$
\frac{y}{N}\Big(\Delta\big( \frac{x}{N} \big)^{\gamma-2}\Big)^{1/2}
$$
in \eqref{4.2} is
\begin{align}\label{4.3}
\nonumber
&\ll Q\frac{x}{N} \underset{\substack{\Delta \leqslant 4HN^{\gamma}Q^{-1} \\ |n_1-n_2| \ll Nx^{\theta-1}}}{\sum_{\substack{h_1 \sim H \\ h_2 \sim H}} \sum_{\substack{n_1 \sim N \\ n_2 \sim N}}} \frac{y}{N}\Big(\Delta\big( \frac{x}{N} \big)^{\gamma-2}\Big)^{1/2} \\
&\ll Q^{\frac{1}{2}}x^{\frac{1}{2}+\theta}N^{-1} \cdot \cD(4HN^{\gamma}Q^{-1}).
\end{align}	
Here the function $\cD(\delta)$ denotes the number of elements of the following set
$$
\{(n_1,n_2,h_1,h_2):n_i \sim N, \ h_i \sim H, \ \big| h_1n^\gamma_1 - h_2n^\gamma_2 \big| \leqslant \delta\},
$$
where we drop the condition $|n_1-n_2| \ll Nx^{\theta-1}$. By Lemma~\ref{lemma:FoIw} we have the contribution of \eqref{4.3} is
\begin{align}\label{4.4}
\ll Q^{-\frac{1}{2}}x^{\frac{5}{2}+\theta-2\gamma}N + Q^{\frac{1}{2}}x^{\frac{3}{2}+\theta -\gamma}.
\end{align}
When $\Delta < y^{-2}x^{2-\gamma}N^{\gamma}$, we have 
$$
\frac{y}{N} < \Big(\Delta\big( \frac{x}{N} \big)^{\gamma-2}\Big)^{-1/2}.
$$ 
Thus the $y/N$ term in the minimum of \eqref{4.2} produces a contribution
\begin{align}\label{4.5}
\ll Q\frac{x}{N} \frac{y}{N} \cdot \cD(y^{-2}x^{2-\gamma}N^{\gamma}).
\end{align}
If $\Delta \geqslant y^{-2}x^{2-\gamma}N^{\gamma}$, we find that the last term $\big(\Delta(x/N)^{\gamma-2}\big)^{-1/2}$ in \eqref{4.2} contributes
\begin{align}\label{4.6}
\ll Q\frac{x}{N} \big( \frac{x}{N} \big)^{1-\frac{\gamma}{2}} \max_{y^{-2}x^{2-\gamma}N^{\gamma} \leqslant \Delta \leqslant 4HN^{\gamma}Q^{-1}} \Delta^{-\frac{1}{2}} \cD(\Delta).
\end{align}
Note that this estimate covers \eqref{4.5} if we take $\Delta = y^{-2}x^{2-\gamma}N^{\gamma}$. Then again, by Lemma~\ref{lemma:FoIw} we obtain that \eqref{4.6} is
\begin{align}\label{4.7}
\ll Qx^{2+\theta-\gamma}N^{-1} + Q^{\frac{1}{2}}x^{\frac{7}{2}-2\gamma}.
\end{align}
Since the condition under the maximum symbol of \eqref{4.6} implies that $1\leqslant Q \leqslant x^{2\theta -1}$, we combine \eqref{4.4} and \eqref{4.7} to derive that
$$
x^{-\varepsilon}|S_{II}|^2 \ll Q^{-\frac{1}{2}}x^{\frac{5}{2}+\theta-2\gamma}N + Q^{\frac{1}{2}}x^{\frac{3}{2}+\theta -\gamma} + Qx^{2+\theta-\gamma}N^{-1} + Q^{\frac{1}{2}}x^{\frac{7}{2}-2\gamma}.
$$
Optimizing $Q$ over $[1,x^{2\theta-1}]$ and using Lemma~\ref{lemmea: balance} we find that type II sum holds if 
\begin{align*}
x^{2-\theta-\gamma +\varepsilon}\ll N \ll x^{3\theta+5\gamma-7-\varepsilon} 
\end{align*}
with 
$$
\gamma > \frac{9-4\theta}{6} \quad \text{and} \quad \gamma > \frac{7-4\theta}{4}.
$$
With $\theta \in (3/4,1]$, it follows that
$$
\frac{9-4\theta}{6} > \frac{7-4\theta}{4}
$$
all the time. We finish the proof of Proposition~\ref{lem:type2_1}.

\section{An asymptotic result: Proof of Theorem \ref{thm1}}\label{sec:An asymptotic result}
In this section, we shall give a proof of Theorem~\ref{thm1}. Let $\theta\in (2/3,1)$ and $I$ be the short interval $(x,x+y]$ as defined. By Lemma~\ref{lem:PS} we obtain that
$$
\pi_{c}(x+y) - \pi_{c}(x) = \Sigma_1 + \Sigma_2 + O(1),
$$
where
\begin{align*}
&\Sigma_1 \defeq \sum_{\substack{ p \in I }} \gamma  p^{\gamma-1}, \\
&\Sigma_2 \defeq \sum_{\substack{ p \in I }} \big( \psi(-(p+1)^{\gamma}) - \psi(-p^{\gamma}) \big).
\end{align*}
The main term is
\begin{align}
\nonumber \Sigma_1 &= \sum_{p \in I}\gamma x^{\gamma-1} + O\Big( \sum_{p \in I} \big((x+y)^{\gamma-1} -x^{\gamma-1}\big) \Big).
\end{align}
Since the well-known result by Huxley \cite{Hux} shows that 
\begin{align*}
	\pi(x+y)-\pi(x)=(1+o(1))\frac{y}{\log x}
\end{align*}
when $y \geqslant x^{7/12+\varepsilon}$, we obtain that
\begin{align}\label{Main term}
\Sigma_1 &= \frac{\gamma yx^{\gamma-1}}{\log x} + O\Big(\frac{yx^{\gamma-1}}{\log^2 x}\Big).
\end{align}
Next we turn our attention to $\Sigma_2$. By \eqref{Main term} and applying the prime number theorem, it is sufficient to show that
\begin{align*}
S \defeq \sum_{\substack{ n \in I }} \Lambda(n) \big( \psi(-(n+1)^{\gamma}) - \psi(-n^{\gamma}) \big) \ll x^{\theta + \gamma - 1 - \varepsilon}.
\end{align*}
According to Vaaler's approximation (Lemma~\ref{lem:Vaaler}), we write that
$$
S = S_1 + O(|S_2|),
$$
where
\begin{align*}
&S_1 \defeq \sum_{1 \leqslant |h| \leqslant H} a_h \sum_{n \in I} \Lambda(n) \Big( \e(h(n+1)^\gamma) - \e(hn^\gamma) \Big), \\
&S_2 \defeq \sum_{0 \leqslant |h| \leqslant H} b_h \sum_{n \in I} \Lambda(n) \Big( \e(h(n+1)^\gamma) + \e(hn^\gamma) \Big).
\end{align*}
Taking $H = x^{1-\gamma+
	\varepsilon}$, by a standard method similar to arguments of Lemma~\ref{harman sieve type I II information}, it suffices to prove that 
$$
\sum_{h \sim H} \Big| \sum_{n\in I} \Lambda(n) \e(hn^\gamma) \Big| \ll x^{\theta-\varepsilon}.
$$ 

Next we combine our exponential sum estimations, including Proposition~\ref{lem:type1}, \ref{lem:type2_2} and \ref{lem:type2_1} by the Heath-Brown identity (see Page 1366-1368 in \cite{HB}). The following lemma can be proved by a same idea to Proposition 1 in \cite{BF}. 
\begin{lemma}\label{lem:HB indentity}
If we have real numbers $0<a<1, 0<b<c<1$ satisfying
$$
b<\frac{2}{3}, \quad 1-c<c-b \quad \text{and}\quad 1-a<\frac{1}{2}c,
$$
then the sum
$$
\sum_{h \sim H} \delta_h \sum_{n\in I} \Lambda(n) \e(hn^\gamma) 
$$
could be bounded by type I sum defined as equation \eqref{eq:tyI} with $N \leqslant x^a$ and type II sum defined as equation \eqref{eq:tyII} with $x^b \leqslant N \leqslant x^c$.
\end{lemma}
We apply Lemma~\ref{lem:HB indentity} to the following several cases. If $\theta \in (2/3,40/51]$, we apply Proposition~\ref{lem:type1} and \ref{lem:type2_2}. We choose
$$
a = 2\gamma-\frac{4}{3}, \quad b=5-\theta-4\gamma \quad \text{and}\quad c=\theta +2\gamma-2. 
$$
By Lemma~\ref{lem:HB indentity}, we know Theorem~\ref{thm1} holds provided that
$$
\left\{\begin{array}{ll}
	5-\theta-4\gamma < \frac{2}{3}, & \\
	1 - (\theta +2\gamma-2) < (\theta +2\gamma-2) - (5-\theta-4\gamma), & \\
	1 - (2\gamma-\frac{4}{3}) < \frac{1}{2} (\theta +2\gamma-2). &
\end{array}\right.
$$
It follows that
\begin{align}\label{condition1}
\frac{20-3\theta}{18} < \gamma < 1.
\end{align}
For $\theta > 3/4$, we use Proposition~\ref{lem:type1}, \ref{lem:type2_2} and \ref{lem:type2_1} flexibly to get the best suitable range of $\gamma$. Since the range of type II is complex, we give Figure~\ref{fig1} for assistance. Here the blue interval is obtained by Proposition~\ref{lem:type2_2} and the red interval is obtained by Proposition~\ref{lem:type2_1}. These intervals intersect under some conditions.
\begin{figure}[htbp]
	\centering
\begin{tikzpicture}[scale=14] 
	\draw (0,0) -- (1,0);
	\draw (0,0.02) -- (0,0);
	\node[below] at (-0.013,0.005) {0};
	\draw (1,0.02) -- (1,0);
	\node[below] at (1.013,0.005) {1};
	
	\node[red, rotate=0] at (0.05, 0) {(};
	\node[red, rotate=0] at (0.21, 0) {)};
	
	\node[red] at (0.05, -0.04) {$2-\theta-\gamma$};
	\node[red] at (0.21, -0.04) {$3\theta+5\gamma-7$};
		
	\node[blue, rotate=0] at (0.29, 0) {(};
	\node[blue, rotate=0] at (0.45, 0) {)};
	
	\node[blue] at (0.29, 0.04) {$3-\theta-2\gamma$};
	\node[blue] at (0.45, 0.04) {$\theta+4\gamma-4$};
	
	\node[blue, rotate=0] at (0.55, 0) {(};
	\node[blue, rotate=0] at (0.71, 0) {)};
	
	\node[blue] at (0.55, -0.04) {$5-\theta-4\gamma$};
	\node[blue] at (0.71, -0.04) {$\theta+2\gamma-2$};
	
	\node[red, rotate=0] at (0.79, 0) {(};
	\node[red, rotate=0] at (0.95, 0) {)};
	
	\node[red] at (0.79, 0.04) {$8-3\theta-5\gamma$};
	\node[red] at (0.95, 0.04) {$\theta+\gamma-1$};
\end{tikzpicture}
\caption{Range of type II sums}\label{fig1}
\end{figure}
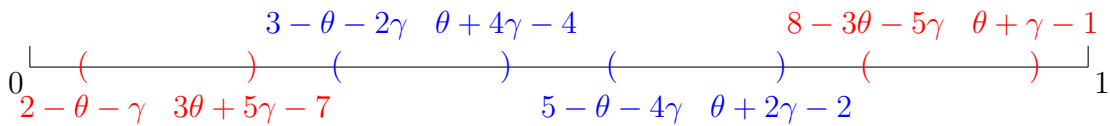

When $\theta \in (10/11,1)$, we suppose that
\begin{align} \label{condition0}
3\theta+5\gamma-7 > 8-3\theta-5\gamma \Rightarrow \frac{15-6\theta}{10} < \gamma < 1
\end{align}
to make sure the rightmost interval and the leftmost interval in Figure~\ref{fig1} intersect. Hence we take
$$
a = 2\gamma-\frac{4}{3}, \quad b=2-\theta-\gamma \quad \text{and}\quad c=\theta +\gamma-1. 
$$
By Proposition~\ref{lem:type1}, \ref{lem:type2_1} and Lemma~\ref{lem:HB indentity}, we have that Theorem~\ref{thm1} holds provided that
\begin{equation*}\label{condition2}
	\left\{\begin{array}{ll}
		2-\theta-\gamma < \frac{2}{3}, & \\
		1 - (\theta +\gamma-1) < (\theta +\gamma-1) - (2-\theta-\gamma), & \\
		1 - (2\gamma-\frac{4}{3}) < \frac{1}{2} (\theta +\gamma-1). &
	\end{array}\right.
\end{equation*}
It follows that
\begin{align}\label{condition3}
\frac{17-3\theta}{15} < \gamma < 1.
\end{align}
Finally, for $\theta \in (40/51,10/11]$, we use both Proposition~\ref{lem:type2_2} and \ref{lem:type2_1}. We suppose that
$$
\left\{\begin{array}{ll}
	3\theta+5\gamma-7 > 3-\theta-2\gamma, & \\
	\theta+4\gamma-4 > 5-\theta-4\gamma, & \\
	\theta+2\gamma-2 > 8-3\theta-5\gamma &
\end{array}\right.
$$
to make sure the four intervals in Figure~\ref{fig1} connect to form one interval. It follows that
\begin{align}\label{condition4}
\max \Big( \frac{9-2\theta}{8} , \frac{10-4\theta}{7} \Big) < \gamma <1.
\end{align}
Hence we also take
$$
a = 2\gamma-\frac{4}{3}, \quad b=2-\theta-\gamma \quad \text{and}\quad c=\theta +\gamma-1 
$$
in this case. By Proposition~\ref{lem:type1}, \ref{lem:type2_2}, \ref{lem:type2_1} and Lemma~\ref{lem:HB indentity}, we have Theorem~\ref{thm1} holds provided that \eqref{condition3} holds. Now combine the three cases and the conditions \eqref{condition of type2_1}, \eqref{condition of type2_2}, \eqref{condition1}, \eqref{condition0}, \eqref{condition3} and \eqref{condition4} to obtain that
$$
\sum_{h \sim H} \Big| \sum_{n\in I} \Lambda(n) \e(hn^\gamma) \Big| \ll x^{\theta-\varepsilon}
$$ 
for $\gamma$ satisfying conditions in \eqref{eq:thm1}. Hence we finish the proof of Theorem~\ref{thm1}.

\section{A lower bound result: Proof of Theorem \ref{thm2}} \label{sec:lower bound}
The aim of this section is to describe how we employ Harman sieve (see \cite{BaHaRi}) to our problem. We combine results on exponential sums in Section~\ref{Estimate of Type I Sum} and \ref{Estimate of Type II Sum} with an alternative sieve. This allows us to get a formula without an asymptotic result but with a larger value of $\gamma$ when $\theta \in (2/3,1)$. In this section, we always suppose that $\theta \in (2/3,1)$.

\subsection{The Fundamental Lemma}
We define
$$
\cA \defeq \{n\in \NN: n \in \cN_c \cap I\} \quad \text{and} \quad \cB \defeq \{ n \in \NN: n \in I \}.
$$
By type I and type II information, we establish a fundamental lemma of Piatetski-Shapiro sequence in short interval similar to \cite[Theorem~3.1]{Ha} or \cite[Lemma~5.4]{SDP}. Since in our problem $\theta>7/12=0.58...$, the proof is done if we convert the set $\cA$ to the set $\cB$ by applying Huxley's theorem to $\cB$. To establish the fundamental lemma, the following lemmas are needed. We mention that the ideas of our arguments follow from \cite[Lemma~7-10]{Kum}.

\begin{lemma}\label{harman sieve type I II information}
Let $\varepsilon>0$ be small, $2/3 < \theta < 1$ and $a_n,b_m \ll x^{\varepsilon}$. For
$$
N \ll x^{2\gamma -\frac{4}{3} -\varepsilon},
$$
we have 
\begin{align}\label{harman sieve type I}
\sum_{n \sim N} \sum_{\substack{m \\ mn \in \cA}} a_n  = \sum_{n \sim N} \sum_{\substack{m \\ mn \in \cB}} a_n \gamma (mn)^{\gamma-1} + O(x^{\theta + \gamma - 1 -\varepsilon}).
\end{align}
provided that \eqref{condition of type1} holds.
Moreover, if
$$
x^{5-\theta-4\gamma +\varepsilon}\ll N \ll x^{\theta+2\gamma-2-\varepsilon} \quad \text{or} \quad x^{3-\theta-2\gamma+\varepsilon} \ll N \ll x^{\theta + 4\gamma -4-\varepsilon},
$$
then
\begin{align}\label{harman sieve type II_2}
	\sum_{n \sim N} \sum_{\substack{m \\ mn \in \cA}} a_n b_m  = \sum_{n \sim N} \sum_{\substack{m \\ mn \in \cB}} a_n b_m \gamma (mn)^{\gamma-1} + O(x^{\theta + \gamma - 1 - \varepsilon}),
\end{align}
provided that \eqref{condition of type2_2} holds. For $3/4 < \theta <1$ and
$$
x^{2-\theta-\gamma +\varepsilon}\ll N \ll x^{3\theta+5\gamma-7-\varepsilon} \quad \text{or} \quad x^{8-3\theta-5\gamma+\varepsilon} \ll N \ll x^{\theta + \gamma -1-\varepsilon},
$$ 
we still have
\begin{align}\label{harman sieve type II_1}
	\sum_{n \sim N} \sum_{\substack{m \\ mn \in \cA}} a_n b_m  = \sum_{n \sim N} \sum_{\substack{m \\ mn \in \cB}} a_n b_m \gamma (mn)^{\gamma-1} + O(x^{\theta + \gamma - 1 - \varepsilon}),
\end{align}
provided that \eqref{condition of type2_1} holds.
\end{lemma}

\begin{proof}
We give a proof on type II sum and omit the proof on type I sum since methods on type I sum are similar and simpler. By Lemma~\ref{lem:PS}, we have
\begin{align*}
\sum_{n \sim N} \sum_{\substack{m \\ mn \in \cA}} a_n b_m &= \sum_{n \sim N} \sum_{\substack{m \\ mn \in \cB}} a_n b_m  \Big( \fl{-(mn)^{\gamma}} - \fl{-(mn+1)^{\gamma}} \Big) \\
&= \Sigma_1 + \Sigma_2,
\end{align*}
where 
\begin{align*}
\Sigma_1 \defeq \sum_{n \sim N} \sum_{\substack{m \\ mn \in \cB}} a_n b_m \gamma (mn)^{\gamma-1} + O(x^{\varepsilon})
\end{align*}
and
\begin{align*}
\Sigma_2 \defeq \sum_{n \sim N} \sum_{\substack{m \\ mn \in \cB}} a_n b_m \Big( \psi(-(mn+1)^{\gamma}) - \psi(-(mn)^{\gamma}) \Big). 
\end{align*}
It is sufficient to prove that $\Sigma_2 \ll x^{\theta + \gamma - 1 - \varepsilon}$. Let
$$
f(k) \defeq  \sum_{n \sim N} \sum_{\substack{m \\ mn = k }} a_n b_m.
$$
By the well-known Fourier expansion (Lemma~\ref{lem:Vaaler}), we write that
\begin{align*}
\Sigma_2 &= \Sigma_{21} + O(\Sigma_{22}),
\end{align*}
where 
\begin{align*}
\Sigma_{21} &\defeq \sum_{1 \leqslant |h| \leqslant H} c_h \sum_{\substack{k \in \cB}} f(k)\big( \e(h(k+1)^{\gamma}) - \e(hk^{\gamma}) \big), \\
\Sigma_{22} &\defeq \sum_{0 \leqslant |h| \leqslant H} c'_h \sum_{\substack{k \in \cB}} f(k)\big( \e(h(k+1)^{\gamma}) + \e(hk^{\gamma}) \big),  
\end{align*}
with $H = x^{1 - \gamma +\varepsilon}$ and $c_h \ll 1/|h|, c'_h \ll 1/H$. 

Firstly, we consider the upper bound of $\Sigma_{21}$. Define
$$
\phi_h(t) \defeq \e\big(h((t+1)^{\gamma}-t^{\gamma})\big)-1. 
$$
It follows that
\begin{align*}
\Sigma_{21} = \sum_{1 \leqslant |h| \leqslant H} c_h \sum_{\substack{k \in \cB}} f(k) \phi_h(k) \e(hk^{\gamma}).
\end{align*}
Combining with the upper bound $c_h \ll 1/|h|$ yields
\begin{align*}
\Sigma_{21} \ll \sum_{1 \leqslant |h| \leqslant H} \frac{1}{|h|} \Big| \sum_{\substack{k \in \cB}} f(k) \phi_h(k) \e(hk^{\gamma}) \Big|.
\end{align*}
It follows from partial summation and the bounds
$$
\phi_h(t) \ll |h|t^{\gamma-1} \quad \text{and} \quad \frac{\partial \phi_h(t)}{\partial t} \ll |h|t^{\gamma-2}
$$
that
\begin{align}\label{sigma21}
\nonumber
\Sigma_{21} &\ll \sum_{1 \leqslant |h| \leqslant H} \frac{1}{|h|} \Big| \int_{x}^{x+y} \phi_h(t) d\Big( \sum_{\substack{x < k \leqslant t}} f(k) \e(hk^{\gamma}) \Big) \Big| \\
\nonumber
&\ll \sum_{1 \leqslant |h| \leqslant H} \frac{1}{|h|} \Big| \phi_h(x+y) \sum_{\substack{x < k \leqslant x+y}} f(k) \e(hk^{\gamma})\Big| \\
\nonumber
&\quad + \int_{x}^{x+y} \sum_{1 \leqslant |h| \leqslant H} \frac{1}{|h|} \Big| \frac{\partial \phi_h(t)}{\partial t} \Big| \Big| \sum_{\substack{x < k \leqslant t}} f(k) \e(hk^{\gamma}) \Big|dt \\
\nonumber
&\ll  x^{\gamma-1} \sum_{1 \leqslant |h| \leqslant H} \Big| \sum_{\substack{x < k \leqslant x+y}} f(k) \e(hk^{\gamma}) \Big| \\
&\quad + x^{\theta+\gamma-2} \max_{x < t \leqslant x+y} \sum_{1 \leqslant |h| \leqslant H} \Big| \sum_{\substack{x < k \leqslant t}} f(k) \e(hk^{\gamma}) \Big|.
\end{align}
Here the last term of \eqref{sigma21} is
$$
\ll x^{\theta+\gamma-2} \max_{x < t \leqslant 2x} \sum_{1 \leqslant |h| \leqslant H} \Big| \sum_{\substack{x < k \leqslant t}} f(k) \e(hk^{\gamma}) \Big|.
$$
According to the previous results of primes in Piatetski-Shapiro sequence, i.e. \cite{RiSa}, we see that the above sum is $O(x^{\theta+\gamma-1-\varepsilon})$ for $2817/2426<\gamma<1$ which is admissible for Theorem~\ref{thm1} and \ref{thm2}. Hence we obtain that 
$$
\Sigma_{21} \ll x^{\gamma-1} \sum_{1 \leqslant |h| \leqslant H} \delta_h \sum_{n \sim N}\sum_{\substack{m \\ x < mn \leqslant x+y }} a_n b_m \e(h(mn)^{\gamma}) + x^{\theta+\gamma-1-\varepsilon}
$$
for some $\delta_h \in \CC$ with $|\delta_h|=1$.
By Proposition~\ref{lem:type2_2} and \ref{lem:type2_1}, we obtain that $\Sigma_{21} = O(x^{\theta + \gamma - 1 - \varepsilon})$. As for $\Sigma_{22}$, the situation is similar to $\Sigma_{21}$ for $|h| \geqslant 1$ in $\Sigma_{22}$. The contribution of $h=0$ of $\Sigma_{22}$ is
$$
\ll \frac{1}{H} \sum_{k \in \cB} f(k) \ll x^{\theta + \gamma - 1 - \varepsilon}.
$$
This completes the proof.
\end{proof}

\begin{lemma}\label{a_mn}
Let $I$, $J$ be integers and $\cI_i$, $\cJ_j$ be intervals for $1 \leqslant i \leqslant I$, $1 \leqslant j \leqslant J$. Write
\begin{align}\label{definition of a_mn}
a_{mn} \defeq \sum_{\substack{kp_1 \cdots p_l = n \\ p_1 < p_2 < \cdots < p_l \\ p_i \in \cI_i}} c_{n}
\sum_{\substack{lq_1 \cdots q_j = m \\ q_1 < q_2 < \cdots < q_j \\ q_j \in \cJ_j}} d_{m}
\end{align}
with $|c_n|$, $|d_m| \leqslant x^{\varepsilon}$ and $p_1, \ldots, p_l$ and $q_1, \ldots, q_j$ satisfying $t$ joint conditions of the form
$$
p_u \leqslant q_v \quad \text{or} \quad q_v \leqslant p_u
$$
or
$$
\prod_{u \in U} p_u\prod_{v \in V} q_v \leqslant H \quad \text{or} \quad \prod_{u \in U} p_u \geqslant \prod_{v \in V} q_v
$$
or similar (for given $U \subset \{1, \ldots, I \}$, $V \subset \{1, \ldots, J\}$ and $H\leqslant x$). Suppose that $\varepsilon, \gamma, \theta, N$ satisfy the conditions of Proposition~\ref{lem:type2_2} or Proposition~\ref{lem:type2_1}, then
$$
\sum_{n \sim N} \sum_{\substack{m \\ mn \in \cA}} a_{mn} = \sum_{n \sim N} \sum_{\substack{m \\ mn \in \cB}} a_{mn} \gamma(mn)^{\gamma-1} + O(x^{\theta + \gamma - 1 -\varepsilon}).
$$
\end{lemma}

\begin{proof}
In order to remove the condition $p_u \leqslant q_v $ we apply Lemma~\ref{lemma:perron} with $\alpha = \log p_u$, $\beta = \log\bigl(q_v + \tfrac{1}{2}\bigr)$ and $T = x^2$. We find
\begin{align*}
	&\quad\frac{1}{\pi} \int_{-T}^{T} \sum_{n \sim N} \sum_{\substack{m \\ mn \in \cA}} a_1(m, n, y) \, \frac{dy}{y} \\
	&= \sum_{n \sim N} \sum_{\substack{m \\ mn \in \cA}} a_1(m,n) \cdot \frac{1}{\pi} \int_{-T}^{T} e^{iy\log p_u} \frac{\sin y(\log(q_v+\frac{1}{2}))}{y} \, dy \\
	&= \sum_{n \sim N} \sum_{\substack{m \\ mn \in \cA}} a_{mn} + O(1).
\end{align*}
where
\begin{equation*}
	a_1(m, n, y) = a_1(m,n) \, p_u^{iy} \sin \bigl(y \log(q_v + \tfrac{1}{2})\bigr)
\end{equation*}
and $ a_1(m, n)$ is the same as $a_{mn}$ but with the condition $p_u \leqslant q_v$ removed. Applying this procedure $t$ times we obtain that
\begin{align}\label{amn to amny}
	\sum_{n \sim N} \sum_{\substack{m \\ mn \in \cA}} a_{mn} = \frac{1}{\pi^t} \int_{-T}^{T} \!\!\cdots\ \!\! \int_{-T}^{T} \sum_{n \sim N} \sum_{\substack{m \\ mn \in \cA}} a^*_1(m, n, \mathbf{y}) \, \frac{dy_1 \cdots dy_t}{y_1 \cdots y_t} + O(1)
\end{align}
where $a^*_1(m, n, \mathbf{y})$ is defined similarly to $a_{mn}$  but with all the joint conditions removed, so that
\begin{equation*}
	a^*_1(m, n, \mathbf{y}) = a(m,\mathbf{y} ) \, b(n, \mathbf{y}),
\end{equation*}
where $a(m,\mathbf{y} )$ and $b(n, \mathbf{y})$ are defined similar to $a_1(m,n,y)$. Therefore, we apply \eqref{harman sieve type II_2} or \eqref{harman sieve type II_1}  to the last sum in \eqref{amn to amny}. We get
\begin{align*}
&\frac{1}{\pi^t} \int_{-T}^{T} \!\!\cdots\ \!\! \int_{-T}^{T} \sum_{n \sim N} \sum_{\substack{m \\ mn \in \cA}}  \frac{a(m, \mathbf{y}) \, b(n, \mathbf{y})}{y_1 \cdots y_t} \, dy_1 \cdots dy_t \\
&= \frac{1}{\pi^t} \int_{-T}^{T} \!\!\cdots\ \!\! \int_{-T}^{T} \sum_{n \sim N} \sum_{\substack{m \\ mn \in \cA}}  \frac{a(m, \mathbf{y}) \, b(n, \mathbf{y})}{y_1 \cdots y_t} \, \gamma(mn)^{\gamma-1} \, dy_1 \cdots dy_t + O\bigl(  x^{\theta+\gamma-1-\varepsilon}\bigr).
\end{align*}
Next applying Lemma~\ref{lemma:perron} $t$ times once again, we finally find
\begin{align*}
&\frac{1}{\pi^t} \int_{-T}^{T} \!\!\cdots\ \!\! \int_{-T}^{T} \sum_{n \sim N} \sum_{\substack{m \\ mn \in \cA}}  \frac{a(m, \mathbf{y}) \, b(n, \mathbf{y})}{y_1 \cdots y_t} \, \gamma(mn)^{\gamma-1} \, dy_1 \cdots dy_t \\
&= \sum_{n \sim N} \sum_{\substack{m \\ mn \in \cA}}  a(m, n) \, \gamma(mn)^{\gamma-1} + O(1).
\end{align*}
\end{proof}

We give the fundamental lemma as followed. Define for $n \in \NN$,
$$
\cS(\cA_n,z) \defeq \sum_{\substack{m \\ mn\in\cA}} \rho(m,z) \quad \text{and} \quad \cS(\cB_n,z) \defeq \sum_{\substack{m \\ mn\in\cB}} \rho(m,z), 
$$
where $\rho(n,z) \defeq \textbf{1}_{p|n \Rightarrow p\geqslant z}$. Thus we have
\begin{lemma}[Fundamental Lemma]\label{The Fundamental Lemma}
Let $u \geqslant 1$ be an integer, $\varepsilon > 0$ be a small positive number. 
	
(i) Let $p_1, \ldots, p_u \in [1, x]$ be such that $\prod_{1 \leqslant k \leqslant u} p_k \leqslant x$. Suppose that $\varepsilon, \gamma, \theta, N$ satisfy the conditions of Proposition~\ref{lem:type2_2} or Proposition~\ref{lem:type2_1} and there exists $\cD \subset \{1, \ldots, u\}$ such that  
$$\prod_{k \in D} p_k \asymp N.$$
Then  
\begin{equation}\label{fundamental lemma 1}
\sum_{\substack{ p_1 , \cdots , p_u  }} \cS(\cA_{p_1, \cdots, p_u}, p_1) = \gamma x^{\gamma-1} \sum_{\substack{p_1 , \cdots , p_u }} \cS(\cB_{p_1, \cdots, p_u}, p_1) + O(x^{\theta + \gamma - 1 - \varepsilon}).
\end{equation}

(ii)For any $a_n \ll x^{\varepsilon}$, if $\varepsilon, \gamma, \theta$ satisfy the conditions of Proposition~\ref{lem:type1} and satisfy the conditions of Proposition~\ref{lem:type2_1} or Proposition~\ref{lem:type2_2} we have
\begin{equation}\label{fundamental lemma 2}
\sum_{n \sim N} a_n \cS(\cA_n, z) = \gamma x^{\gamma-1} \sum_{n \sim N} a_n \cS(\cB_n, z) + O(x^{\theta + \gamma - 1 - \varepsilon} ),
\end{equation}
where 
$$
N \ll x^{3\theta + 5\gamma -7 -\varepsilon} \quad \text{and} \quad z = x^{4\theta+6\gamma-9-\varepsilon}
$$ 
for $\varepsilon, \gamma, \theta$ satisfying the conditions of Proposition~\ref{lem:type1} and \ref{lem:type2_1}, while 
$$
N \ll x^{\theta + 4\gamma -4-\varepsilon} \quad \text{and} \quad z = x^{2\theta+6\gamma-7-\varepsilon}
$$ 
when $\varepsilon, \gamma, \theta$ satisfy the conditions of Proposition~\ref{lem:type1} and Proposition~\ref{lem:type2_2}.
\end{lemma}

\begin{proof}
We prove \eqref{fundamental lemma 1} first. By the definition of $\cS(\cA_n,z)$, we have
\begin{align*}
\sum_{\substack{p_1 , \cdots , p_u}} \cS(\cA_{p_1, \cdots, p_u}, p_1) = \sum_{\substack{p_1 , \cdots , p_u}} \sum_{\substack{s \\ p_1\cdots p_u s \in \cA \\ (s, P(p_1))=1}} 1.
\end{align*}
Let 
$$
n = \prod_{j\in \cD} p_j, \quad m = s\prod_{j \notin \cD} p_j.
$$
Hence we arrive at
\begin{align} \label{fundamental lemma 3}
\nonumber
&\sum_{\substack{p_1 , \cdots , p_u}} \cS(\cA_{p_1, \cdots, p_u}, p_1) \\
\nonumber
& = \sum_{n \sim N} \sum_{\substack{m \\ mn \in \cA \\ (m,P(p_1)) = 1}} \bigg( \sum_{\substack{\prod_{j\in \cD} p_j = n}} 1 \bigg) \bigg( \sum_{ s\prod_{j \notin \cD} p_j = m} 1 \bigg) \\
\nonumber
& = \sum_{n \sim N} \sum_{\substack{m \\ mn \in \cA }} \bigg( \sum_{\substack{\prod_{j\in \cD} p_j = n}} 1 \bigg) \bigg( \sum_{ s\prod_{j \notin \cD} p_j = m} \sum_{d|(m,P(p_1))} \mu(d) \bigg)\\
& \eqdef \sum_{n \sim N} \sum_{\substack{m \\ mn \in \cA }} a_{mn},
\end{align}
where we use the identity
$$
\textbf{1}_{(m,n)=1} = \sum_{d|(m,n)} \mu(d) .
$$
Noting that $a_{mn}$ in \eqref{fundamental lemma 3} satisfies the definition of \eqref{definition of a_mn}, we apply Lemma~\ref{a_mn} to \eqref{fundamental lemma 3} and find that
\begin{align}\label{fundamental lemma 4}
\sum_{\substack{p_1 , \cdots , p_u}} \cS(\cA_{p_1, \cdots, p_u}, p_1) = \sum_{n \sim N} \sum_{\substack{m \\ mn \in \cB}} a_{mn} \gamma(mn)^{\gamma-1} + O(x^{\theta + \gamma - 1 -\varepsilon}).
\end{align} 
The proof is finished if we prove that the double sum of the right-hand side in \eqref{fundamental lemma 4} is
$$
\gamma x^{\gamma-1} \sum_{\substack{p_1 , \cdots , p_u}} \cS(\cB_{p_1, \cdots, p_u}, p_1) + O(x^{\theta + \gamma - 1 - \varepsilon}).
$$ 
Since $a_{mn} \geqslant 0$, we subtract the two main terms and use the trivial bound to get
\begin{align*}
\nonumber
&\sum_{n \sim N} \sum_{\substack{m \\ mn \in \cB}} a_{mn} \gamma(mn)^{\gamma-1} - \gamma x^{\gamma-1} \sum_{\substack{p_1 , \cdots , p_u}} \cS(\cB_{p_1, \cdots, p_u}, p_1) \\ 
\nonumber
=&\gamma \sum_{n \sim N} \sum_{\substack{m \\ mn \in \cB}} a_{mn} ((mn)^{\gamma-1}-x^{\gamma-1}) \\
\leqslant&  \gamma ( x^{\gamma-1} - (x+y)^{\gamma-1})\sum_{n \sim N} \sum_{\substack{m \\ mn \in \cB}} a_{mn} \ll y^2x^{\gamma-2} \ll x^{\theta +\gamma -1 - \varepsilon}.
\end{align*} 

As for \eqref{fundamental lemma 2}, we take $z = x^{4\theta + 6\gamma - 9 - \varepsilon}$ and suppose that $\varepsilon, \gamma, \theta, N$ satisfy the conditions of Proposition~\ref{lem:type1} and \ref{lem:type2_1}. We note that
$$
\sum_{n \sim N} a_n \cS(\cA_n, z) = \sum_{n \sim N} a_n \sum_{\substack{ m \\ mn\in \cA \\ (m,P(z))=1}}1 =\sum_{n \sim N} a_n \sum_{\substack{ d|P(z) \\ mnd \in \cA}} \mu(d),
$$
where we change the variable $m$ by $md$ in the last equality. Once we show that
\begin{align}\label{fundamental lemma 5}
\sum_{n \sim N} a_n \sum_{\substack{ d|P(z) \\ mnd \in \cA}} \mu(d) = \sum_{n \sim N} a_n \sum_{\substack{ d|P(z) \\ mnd \in \cB}} \mu(d) \gamma (mnd)^{\gamma-1} + O(x^{\theta + \gamma - 1 - \varepsilon}),
\end{align}
by the similar arguments of \eqref{fundamental lemma 1}, the proof is done. We divide the sum in the left-hand side of \eqref{fundamental lemma 5} into two parts
\begin{align*}
\Sigma_1 \defeq \sum_{n \sim N} a_n \sum_{\substack{ d|P(z) \\ mnd \in \cA \\ nd \leqslant x^{3\theta+5\gamma-7-\varepsilon}}} \mu(d), \quad
\Sigma_2 \defeq \sum_{n \sim N} a_n \sum_{\substack{ d|P(z) \\ mnd \in \cA \\ nd > x^{3\theta+5\gamma-7-\varepsilon}}} \mu(d).
\end{align*}
In $\Sigma_1$ we produce a new variable $k = nd$ and get
$$
\Sigma_1 = \sum_{k \leqslant x^{3\theta+5\gamma-7-\varepsilon}} \sum_{\substack{m \\ km \in \cA }} b_k, \quad \text{where} \quad b_k \defeq \sum_{\substack{nd=k \\ n \sim N \\ d|P(z)}}a_n \mu(d).
$$
Obviously $|b_k| \ll x^{\varepsilon}$, so $\Sigma_1$ is a type I sum. By \eqref{harman sieve type I} in Lemma~\ref{harman sieve type I II information}, we have that
$$
\Sigma_1 = \sum_{n \sim N} a_n \sum_{\substack{ d|P(z) \\ mnd \in \cB \\ nd \leqslant x^{3\theta+5\gamma-7-\varepsilon}}} \mu(d) \gamma (mnd)^{\gamma-1} + O(x^{\theta + \gamma - 1 - \varepsilon}).
$$
Now we write $p_1$ as the largest prime factor of $d$ and replace $d$ by $pd$ to obtain that
\begin{align}\label{fundamental lemma 6}
\Sigma_2 = \sum_{n \sim N} a_n \sum_{p_1 < z} \sum_{\substack{p_1d|P(z) \\ mnp_1d \in \cA \\ np_1d > x^{3\theta+5\gamma-7-\varepsilon}}} \mu(p_1d) = - \sum_{n \sim N}a_n \sum_{p_1 < z} \sum_{\substack{d|P(p_1) \\ mnp_1d \in \cA \\ np_1d > x^{3\theta+5\gamma-7-\varepsilon}}} \mu(d).
\end{align}
We divide the last sum of \eqref{fundamental lemma 6} into two parts , say
\begin{align*}
&\Sigma_3 \defeq \sum_{n \sim N}a_n \sum_{p_1 < z} \sum_{\substack{d|P(p_1) \\ mnp_1d \in \cA \\ nd \leqslant x^{3\theta+5\gamma-7-\varepsilon} < np_1d}} \mu(d) , \\
&\Sigma_4 \defeq \sum_{n \sim N}a_n \sum_{p_1 < z} \sum_{\substack{d|P(p_1) \\ mnp_1d \in \cA \\ nd > x^{3\theta+5\gamma-7-\varepsilon}}} \mu(d).
\end{align*}
As for $\Sigma_3$, we draw two new variables, $k = nd$ and $l=mp_1$. Since $p_1 < z$, we note that $p_1k > x^{3\theta+5\gamma-7-\varepsilon} \Rightarrow k > x^{2 - \theta - \gamma +\varepsilon}.$ Hence similar to the arguments of \eqref{fundamental lemma 3}, we have that $\Sigma_3$ is a type II sum in the form of 
$$
\Sigma_3 = \sum_{\substack{ x^{2-\theta -\gamma + \varepsilon} \leqslant k \leqslant x^{3\theta+5\gamma-7-\varepsilon}}} \sum_{\substack{ l \\ kl\in \cA}} a_{kl},
$$ 
where 
$$
a_{kl} \defeq  \sum_{\substack{ nd = k \\ n \sim N \\ d|P(p_1)}} a_n \mu(d)\sum_{\substack{mp_1 = l \\ p_1 < z}}1.
$$
By Lemma~\ref{a_mn}, we have
$$
\Sigma_3 = \sum_{n \sim N}a_n \sum_{p_1 < z} \sum_{\substack{d|P(p_1) \\ mnp_1d \in \cB \\ nd \leqslant x^{3\theta+5\gamma-7-\varepsilon} < np_1d}} \mu(d) \gamma (mnp_1d)^{\gamma-1} + O(x^{\theta + \gamma - 1 - \varepsilon}).
$$
We treat $\Sigma_4$ similarly to $\Sigma_2$ and write it as the sum of the following two sums
\begin{align*}
&\Sigma_5 \defeq \sum_{n \sim N}a_n \sum_{p_2 < p_1 < z} \sum_{\substack{d|P(p_2) \\ mnp_2p_1d \in \cA \\ nd \leqslant x^{3\theta+5\gamma-7-\varepsilon} < np_2d}} \mu(d) , \\
&\Sigma_6 \defeq \sum_{n \sim N}a_n \sum_{p_2 < p_1 < z} \sum_{\substack{d|P(p_2) \\ mnp_2p_1d \in \cA \\ nd > x^{3\theta+5\gamma-7-\varepsilon}}} \mu(d).
\end{align*}
We deal with $\Sigma_5$ as we did with $\Sigma_3$ to obtain a similar asymptotic formula and we give further decomposition for $\Sigma_6$. We can continue in this treatment to obtain each $\Sigma_{2j-1}$ for which we can apply Lemma~\ref{a_mn} and $\Sigma_{2j}$ for which we give further decomposition. Since the integers in the interval $(x, x+y]$ have $< \log x$ prime divisors after at most $\log x$ such steps we will obtain an empty $\Sigma_{2j}$. Thus we have given asymptotic formula for all the occurring sums. Clearly, combining the asymptotic formula for all $\Sigma_{2j-1}$ we complete the proof of \eqref{fundamental lemma 5} for $N \leqslant x^{3\theta +5\gamma-7 -\varepsilon}$.

When  $\varepsilon, \gamma, \theta, N$ satisfy the conditions of Proposition~\ref{lem:type1} and \ref{lem:type2_2}, the proof is the same. 
\end{proof}

\subsection{Harman Sieve}\label{Harman sieve}
In this section, we repeatedly use the Buchstab identity together with asymptotic formulas in Lemma~\ref{The Fundamental Lemma} to decompose $\cS(\cA,z)$ such that the loss (i.e. the sums discarded) is as small as possible. Here we introduce an important function called Buchstab’s function $w(u)$ (see more details in \cite[Page~14-17]{Ha}) which is defined to be the continuous solution of the differential-difference equation
\begin{align}\label{buchstab equation}
	\left\{
	\begin{aligned}
		w(u) &= \frac{1}{u}, && \text{for } 1 < u \leqslant 2, \\
		(uw(u))' &= w(u-1), && \text{for } u > 2.
	\end{aligned}
	\right.
\end{align} 

We first consider the case when $\theta \in (2/3,0.87]$. Pick $z=x^{2\theta+6\gamma-7}$ and the range of type I and type II sums includes
$$
\Omega_I = [1,x^{\theta+4\gamma-4-\varepsilon}]
$$
and
\begin{align*}
	\Omega_{II} = [x^{3-\theta-2\gamma+\varepsilon},x^{\theta+4\gamma-4-\varepsilon}] \cup [x^{5-\theta-4\gamma+\varepsilon},x^{\theta+2\gamma-2-\varepsilon}] 
\end{align*}
by Proposition~\ref{lem:type1} and \ref{lem:type2_2}. We first apply Buchstab’s identity twice and get
\begin{align}\label{s1s2s3}
\nonumber	\cS(\cA,(2x)^{1/2}) &= \cS(\cA,z) - \sum_{z \leqslant p \leqslant (2x)^{1/2}} \cS(\cA_p,z) + \sum_{z \leqslant p_2 < p_1 \leqslant (2x)^{1/2}} \cS(\cA_{p_1p_2},p_2) \\
	& \eqdef \cS_1 - \cS_2 + \cS_3.
\end{align}
Equation \eqref{fundamental lemma 2} in Lemma~\ref{The Fundamental Lemma} provides asymptotic formulas for $\cS_1$ and $\cS_2$. As for $\cS_3$, noting that not all $p_1, p_2$ belong to the range of type II sum, say \eqref{fundamental lemma 1} in Lemma~\ref{The Fundamental Lemma}, we split the applicable range and discard the rest. Hence we have the following further decomposition.
\begin{align}\label{s4s5}
\nonumber	\cS_3 &= \sum_{\substack{z \leqslant p_2 < p_1 \leqslant (2x)^{1/2} \\ \{p_1,p_2,p_1p_2\} \cap \Omega_{II} \neq \emptyset}} \cS(\cA_{p_1p_2},p_2) + \sum_{\substack{z \leqslant p_2 < p_1 \leqslant (2x)^{1/2} \\ \{p_1,p_2,p_1p_2\} \cap \Omega_{II} =                                                                    \emptyset}} \cS(\cA_{p_1p_2},p_2) \\
	& \eqdef \cS_4 + \cS_5. 
\end{align}
For $\cS_4$ we use \eqref{fundamental lemma 1} in Lemma~\ref{The Fundamental Lemma} and we decompose
\begin{align}\label{s6tos12}
\nonumber \cS_5 &= \sum_{\substack{z \leqslant p_2 < p_1 < x^{3-\theta-2\gamma+\varepsilon} \\ p_1p_2 < x^{3-\theta-2\gamma+\varepsilon}}} \cS(\cA_{p_1p_2},p_2)+ \sum_{\substack{z \leqslant p_2 < p_1 < x^{3-\theta-2\gamma+\varepsilon} \\ x^{\theta + 4\gamma -4 - \varepsilon} < p_1p_2 < x^{5-\theta-4\gamma+\varepsilon}}} \cS(\cA_{p_1p_2},p_2) \\
\nonumber &+ \sum_{\substack{z \leqslant p_2 < p_1 < x^{3-\theta-2\gamma+\varepsilon} \\ p_1p_2 > x^{\theta+2\gamma-2 - \varepsilon}}} \cS(\cA_{p_1p_2},p_2) + \sum_{\substack{z \leqslant p_2 < x^{3-\theta-2\gamma+\varepsilon} \\ x^{\theta + 4\gamma -4 - \varepsilon} < p_1 < x^{5-\theta-4\gamma+\varepsilon} \\x^{\theta + 4\gamma -4 - \varepsilon} < p_1p_2 < x^{5-\theta-4\gamma+\varepsilon}}} \cS(\cA_{p_1p_2},p_2) \\
\nonumber &+ \sum_{\substack{z \leqslant p_2 < x^{3-\theta-2\gamma+\varepsilon} \\ x^{\theta + 4\gamma -4 - \varepsilon} < p_1 < x^{5-\theta-4\gamma+\varepsilon} \\p_1p_2 > x^{\theta + 2\gamma -2 - \varepsilon} }} \cS(\cA_{p_1p_2},p_2) + \sum_{\substack{x^{\theta + 4\gamma -4 - \varepsilon} < p_2 < p_1 < x^{5-\theta-4\gamma+\varepsilon} \\ x^{\theta + 4\gamma -4 - \varepsilon} < p_1p_2 < x^{5-\theta-4\gamma+\varepsilon}}} \cS(\cA_{p_1p_2},p_2) \\
\nonumber &+ \sum_{\substack{x^{\theta + 4\gamma -4 - \varepsilon} < p_2 < p_1 < x^{5-\theta-4\gamma+\varepsilon} \\ p_1p_2 > x^{\theta+2\gamma-2 - \varepsilon}}} \cS(\cA_{p_1p_2},p_2) \\
&\eqdef \cS_6 + \cS_7 +\cS_8+\cS_9+\cS_{10}+\cS_{11}+\cS_{12}.
\end{align}
From \eqref{s1s2s3}, \eqref{s4s5} and \eqref{s6tos12}, we deduce that
$$
\cS(\cA,(2x)^{1/2}) = \cS_1 - \cS_2 + \cS_4 + \cS_6 + \cS_7 +\cS_8+\cS_9+\cS_{10}+\cS_{11}+\cS_{12},
$$
where $\cS_j$ can be evaluated asymptotically, except for $j=6,7, \cdots,12$. Obviously the same decomposition also holds for $\cS(\cB,(2x)^{1/2})$, that is,
$$
\cS(\cB,(2x)^{1/2}) = \cS'_1 - \cS'_2 + \cS'_4 + \cS'_6 + \cS'_7 + \cS'_8+\cS'_9+\cS'_{10}+\cS'_{11}+\cS'_{12},
$$
where $\cS'_j$ is defined similarly to $\cS_j$ with the only difference that $\cA$ is replaced by $\cB$. Since $\cS_j = \gamma x^{\gamma-1} \cS'_j(1+o(1))$ except for $j=6,7,\cdots,12$, we obtain that
\begin{align} \label{s100}
\nonumber \cS(\cA,(2x)^{1/2}) &= \gamma x^{\gamma-1} (1+o(1)) \Big( \cS(\cB,(2x)^{1/2}) - \cS'_6 - \cS'_7 - \cS'_8 - \cS'_9 - \cS'_{10} - \cS'_{11} - \cS'_{12}\Big)  \\
\nonumber &\quad + \cS_6 + \cS_7 + \cS_8+\cS_9+\cS_{10}+\cS_{11}+\cS_{12} \\
&\geqslant \gamma x^{\gamma-1} (1+o(1)) \Big( \cS(\cB,(2x)^{1/2}) - \cS'_6 - \cS'_7 - \cS'_8 - \cS'_9 - \cS'_{10} - \cS'_{11} - \cS'_{12}\Big).
\end{align}
Therefore, the remaining task is to calculate the contributions of $\cS'_6 , \cS'_7$ to $\cS'_{12}$. The calculation is standard and we refer interested readers to \cite[Page~14-17]{Ha}. 

We calculate $\cS'_6 , \cS'_7$ to $\cS'_{12}$ with $\theta=0.80$ and $\gamma=0.9188$ as an example. Let us first define the corresponding region $\cD_6, \cD_7, \cdots ,\cD_{12}$ to be
\begin{align*}
&\cD_6 \defeq \{ (\alpha,\beta) : 0.1128 < \alpha < \beta < 0.3624, 0 < \alpha + \beta <0.3624 \},\\
&\cD_7 \defeq \{ (\alpha,\beta) : 0.1128 < \alpha < \beta < 0.3624, 0.4752 < \alpha + \beta < 0.5248 \},\\
&\cD_8 \defeq \{ (\alpha,\beta) : 0.1128 < \alpha < \beta < 0.3624, 0.6376 < \alpha + \beta < 1 \},\\
&\cD_9 \defeq \{ (\alpha,\beta) : 0.1128 < \alpha < 0.3624, 0.4752 < \beta < 0.5248, 0.4752 < \alpha + \beta < 0.5248 \},\\
&\cD_{10} \defeq \{ (\alpha,\beta) : 0.1128 < \alpha < 0.3624, 0.4752 < \beta < 0.5248, 0.6376 < \alpha + \beta < 1 \},\\
&\cD_{11} \defeq \{ (\alpha,\beta) : 0.4752 < \alpha < \beta < 0.5248, 0.4752 < \alpha + \beta < 0.5248 \},\\
&\cD_{12} \defeq \{ (\alpha,\beta) : 0.4752 < \alpha < \beta < 0.5348, 0.6376 < \alpha + \beta < 1 \}.
\end{align*}
Following the arguments of \cite[Page~14-17]{Ha}, we have
\begin{align*}
&\cS'_6 = I_6 \cS(\cB,(2x)^{1/2}) + O\Big(\frac{y}{\log ^2x}\Big), \\
&\quad \vdots \\
&\cS'_{12} = I_{12} \cS(\cB,(2x)^{1/2}) + O\Big(\frac{y}{\log ^2x}\Big),
\end{align*}
where
\begin{align*}
&I_6 = \int\int_{\cD_6} \omega\Big(\frac{1-\alpha-\beta}{\alpha}\Big)\frac{1}{\alpha^2 \beta} \, d\alpha d\beta, \\
&\quad \vdots \\
&I_{12} = \int\int_{\cD_{12}} \omega\Big(\frac{1-\alpha-\beta}{\alpha}\Big)\frac{1}{\alpha^2 \beta} \, d\alpha d\beta,
\end{align*}
and $\omega(x)$ is Buchstab function defined as \eqref{buchstab equation}.
Using a MATLAB code, we calculate that
\begin{align}\label{s6'tos12'}
\nonumber&\cS'_6 \leqslant 0.6485\cS(\cB,(2x)^{1/2}), \quad \cS'_7 \leqslant 0.1951\cS(\cB,(2x)^{1/2}), \\
\nonumber&\cS'_8 \leqslant 0.0350\cS(\cB,(2x)^{1/2}), \quad \cS'_9 =0, \\
\nonumber&\cS'_{10} \leqslant 0.1077\cS(\cB,(2x)^{1/2}), \quad \cS'_{11} =0, \\
&\cS'_{12} \leqslant 0.0029\cS(\cB,(2x)^{1/2}).
\end{align} 
By \eqref{s100} and \eqref{s6'tos12'}, we finally arrive at
$$
\cS(\cA,(2x)^{1/2}) \geqslant 0.0108\gamma \frac{yx^{\gamma-1}}{\log x}
$$
with $\theta=0.80$ and $\gamma=0.9188$.

Next we consider the case when $\theta \in (0.87,1)$, we pick $z=x^{4\theta+6\gamma-9}$ and the range of type I and type II sums includes
$$
\Omega_I = [1,x^{3\theta+5\gamma-7-\varepsilon}]
$$
and
\begin{align*}
	\Omega_{II} = &[x^{2-\theta-\gamma+\varepsilon},x^{3\theta+5\gamma-7-\varepsilon}] \cup [x^{8-3\theta-5\gamma+\varepsilon},x^{\theta+\gamma-1-\varepsilon}].
\end{align*}
under the conditions of \eqref{condition of type1}, \eqref{condition of type2_1}. We first apply Buchstab’s identity twice and get
\begin{align}\label{f1f2f3}
\nonumber	\cS(\cA,(2x)^{1/2}) &= \cS(\cA,z) - \sum_{z \leqslant p \leqslant (2x)^{1/2}} \cS(\cA_p,z) + \sum_{z \leqslant p_2 < p_1 \leqslant (2x)^{1/2}} \cS(\cA_{p_1p_2},p_2) \\
	& \eqdef \cF_1 - \cF_2 + \cF_3.
\end{align}
Similarly, \eqref{fundamental lemma 2} in Lemma~\ref{The Fundamental Lemma} provides asymptotic formulas for $\cF_1$ and $\cF_2$. We split the applicable range and discard the rest for $\cF_3$. Hence we have the following further decomposition.
\begin{align}\label{f4f5}
\nonumber	\cF_3 &= \sum_{\substack{z \leqslant p_2 < p_1 \leqslant (2x)^{1/2} \\ \{p_1,p_2,p_1p_2\} \cap \Omega_{II} \neq \emptyset}} \cS(\cA_{p_1p_2},p_2) + \sum_{\substack{z \leqslant p_2 < p_1 \leqslant (2x)^{1/2} \\ \{p_1,p_2,p_1p_2\} \cap \Omega_{II} =                                                                    \emptyset}} \cS(\cA_{p_1p_2},p_2) \\
& \eqdef \cF_4 + \cF_5. 
\end{align}
For $\cS_4$ we use \eqref{fundamental lemma 1} in Lemma~\ref{The Fundamental Lemma}. Next we decompose $\cF_5$ as
\begin{align}\label{f6tof15}
\nonumber	\cF_5 &= \sum_{\substack{z \leqslant p_2 < p_1 < x^{2-\theta-\gamma+\varepsilon} \\ p_1p_2 < x^{2-\theta-\gamma+\varepsilon}}} \cS(\cA_{p_1p_2},p_2)+ \sum_{\substack{z \leqslant p_2 < p_1 < x^{2-\theta-\gamma+\varepsilon} \\  x^{3\theta+5\gamma-7-\varepsilon} < p_1p_2 < x^{8-3\theta-5\gamma+\varepsilon}}} \cS(\cA_{p_1p_2},p_2) \\
\nonumber &\quad + \sum_{\substack{z \leqslant p_2 < p_1 < x^{2-\theta-\gamma+\varepsilon} \\ x^{\theta+\gamma-1-\varepsilon} < p_1p_2 < x}} \cS(\cA_{p_1p_2},p_2) + \sum_{\substack{z \leqslant p_2 < x^{2-\theta-\gamma+\varepsilon} \\ x^{3\theta+5\gamma-7-\varepsilon} < p_1 < (2x)^{1/2} \\ x^{3\theta+5\gamma-7-\varepsilon} < p_1p_2 < x^{8-3\theta-5\gamma+\varepsilon}}} \cS(\cA_{p_1p_2},p_2) \\
\nonumber &\quad + \sum_{\substack{z \leqslant p_2 < x^{2-\theta-\gamma+\varepsilon} \\ x^{3\theta+5\gamma-7-\varepsilon} < p_1 < (2x)^{1/2} \\ x^{\theta+\gamma-1-\varepsilon} < p_1p_2 < x}} \cS(\cA_{p_1p_2},p_2) + \sum_{\substack{x^{3\theta+5\gamma-7-\varepsilon} \leqslant p_2 < p_1 < (2x)^{1/2} \\ x^{3\theta+5\gamma-7-\varepsilon} < p_1p_2 < x^{8-3\theta-5\gamma+\varepsilon}}} \cS(\cA_{p_1p_2},p_2)\\ 
\nonumber &\quad + \sum_{\substack{x^{3\theta+5\gamma-7-\varepsilon} \leqslant p_2 < p_1 < (2x)^{1/2} \\ x^{\theta+\gamma-1-\varepsilon} < p_1p_2 < x}} \cS(\cA_{p_1p_2},p_2) \\
&\eqdef \cF_6 + \cF_7 +\cF_8+\cF_9+\cF_{10}+\cF_{11}+\cF_{12}.
\end{align}
From \eqref{f1f2f3}, \eqref{f4f5} and \eqref{f6tof15}, we deduce that
$$
\cS(\cA,(2x)^{1/2}) = \cF_1 - \cF_2 + \cF_4 + \cF_6 + \cF_7 +\cF_8+\cF_9+\cF_{10}+\cF_{11}+\cF_{12},
$$
where $\cS_j$ can be evaluated asymptotically, except for $j=6,7, \cdots,12$. Same as before, we finish proof after using the same decomposition to $\cS(\cB,(2x)^{1/2})$ and the standard process of Harman sieve. 

Here we calculate the case when $\theta = 0.90$ and $\gamma = 0.9168$ as an example and give the results calculated by the MATLAB code. First define the corresponding region $\cT_6, \cT_7, \cdots, \cT_{12}$ to be
\begin{align*}
	&\cT_6 \defeq \{ (\alpha,\beta) : 0.1008 \leqslant \alpha < \beta < 0.1832, 0 < \alpha + \beta < 0.1832 \},\\
	&\cT_7 \defeq \{ (\alpha,\beta) : 0.1008 \leqslant \alpha < \beta < 0.1832, 0.284 < \alpha + \beta < 0.716 \},\\
	&\cT_8 \defeq \{ (\alpha,\beta) : 0.1008 \leqslant \alpha < \beta < 0.1832, 0.8186 < \alpha + \beta < 1 \},\\
	&\cT_9 \defeq \{ (\alpha,\beta) : 0.1008 \leqslant \alpha < 0.1832, 0.284 < \beta < 0.5, 0.284 < \alpha + \beta < 0.716 \},\\
	&\cT_{10} \defeq \{ (\alpha,\beta) : 0.1008 \leqslant \alpha < 0.1832, 0.284 < \beta < 0.5, 0.8186 < \alpha + \beta < 1 \},\\
	&\cT_{11} \defeq \{ (\alpha,\beta) : 0.284 \leqslant \alpha < \beta < 0.5, 0.284 < \alpha + \beta < 0.716 \},\\
	&\cT_{12} \defeq \{ (\alpha,\beta) : 0.284 \leqslant \alpha < \beta < 0.5, 0.8186 < \alpha + \beta < 1 \}.
\end{align*}
Following the arguments of \cite[Page~14-17]{Ha}, we have
\begin{align*}
	&\cF'_6 = I_6 \cS(\cB,(2x)^{1/2}) + O\Big(\frac{y}{\log ^2x}\Big), \\
	&\vdots \\
	&\cF'_{12} = I_{12} \cS(\cB,(2x)^{1/2}) + O\Big(\frac{y}{\log ^2x}\Big),
\end{align*}
where
\begin{align*}
	&I_6 = \int\int_{\cT_6} \omega\Big(\frac{1-\alpha-\beta}{\alpha}\Big)\frac{1}{\alpha^2 \beta} \, d\alpha d\beta, \\
	&\vdots \\
	&I_{12} = \int\int_{\cT_{12}} \omega\Big(\frac{1-\alpha-\beta}{\alpha}\Big)\frac{1}{\alpha^2 \beta} \, d\alpha d\beta,
\end{align*}
and $\omega(x)$ is Buchstab function defined as \eqref{buchstab equation}.
Using a MATLAB code, we calculate that
\begin{align}\label{f6'tof12'}
	\nonumber&\cF'_6 =0, \quad \cF'_7 \leqslant 0.2394\cS(\cB,(2x)^{1/2}), \\
	\nonumber&\cF'_8 =0, \quad \cF'_9 \leqslant 0.6004\cS(\cB,(2x)^{1/2}), \\
	\nonumber&\cF'_{10} =0, \quad \cF'_{11} \leqslant 0.0973\cS(\cB,(2x)^{1/2}), \\
	&\cF'_{12} \leqslant 0.0532\cS(\cB,(2x)^{1/2}).
\end{align} 
Thus by \eqref{f6'tof12'}, we finally arrive at
$$
\cS(\cA,(2x)^{1/2}) \geqslant 0.0097\gamma \frac{yx^{\gamma-1}}{\log x}
$$
with $\theta=0.90$ and $\gamma=0.9168$.

\section*{Acknowledgements}	
This article is supported by the National Natural Science Foundation of China (No. 11901447, 12271422).


\begin{thebibliography}{99}
		
		\bibitem{BHP}
		R. C. Baker, G. Harman and J. Pintz,
		The difference between consecutive primes. II.
		\emph{Proc. London Math. Soc. (3)} 83 (2001), no. 3, 532--562.
		
		\bibitem{BaHaRi}
		R.~C. Baker, G. Harman and J. Rivat, 
		Primes of the form $[n^c]$. 
		\emph{J. Number Theory} 50 (1995), no.~2, 261--277.
		
		\bibitem{BF}
		A. Balog and J.~B. Friedlander, 
		A hybrid of theorems of Vinogradov and Piatetski-Shapiro, 
		\emph{Pacific J. Math.} {\bf 156} (1992), no.~1, 45--62.
		
		\bibitem{FoIw}
		\'E. Fouvry and H. Iwaniec,
		Exponential sums with monomials. 
		\emph{J. Number Theory} 33 (1989), no.~3, 311--333.
		

		\bibitem{GraKol}
		S.~W.~Graham and G.~Kolesnik,
		\emph{Van der Corput's method of exponential sums}.
		London Mathematical Society Lecture Note Series, 126.
		Cambridge University Press, Cambridge, 1991.
		
		\bibitem{GGL}
		L. Guo, V. Z. Guo and L. Lu,
		The Piatetski-Shapiro prime number theorem.
		arXiv:2505.10391.
		
		\bibitem{GM2025}
		L. Guth and J. Maynard,
		New large value estimates for Dirichlet polynomials.
		arXiv 2405.20552
		
		\bibitem{Ha}
		G. Harman, 
		\emph{Prime-detecting sieves}.
		London Mathematical Society Monographs Series, 33. 
		Princeton Univ. Press, Princeton, NJ, 2007.
		
		\bibitem{HB}
		D.~R. Heath-Brown, Prime numbers in short intervals and a generalized Vaughan identity,
		\emph{Canadian J. Math.} {\bf 34} (1982), no.~6, 1365--1377.
		
		\bibitem{HB2}
		D.~R.~Heath-Brown,
		The Pjatecki\u{i}-\u{S}apiro prime number theorem. 
		\emph{J.\ Number Theory} 16 (1983), 242--266.
		
		\bibitem{Hoh1930}
		G. Hoheisel, 
		Trimzahlprobleme in der Analysis. 
		\emph{Sitz. Preuss. Akad. Wiss.} 2 (1930) 1--13.


		
		\bibitem{Hux}
		M.~N. Huxley, 
		On the difference between consecutive primes. 
		\emph{Invent. Math.} 15 (1972), 164--170.
		
		\bibitem{Kum}
		A.~V. Kumchev, 
		On the distribution of prime numbers of the form $[n^c]$. 
		\emph{Glasg. Math. J.} 41 (1999), no.~1, 85--102.
		

		\bibitem{PS}
		I.~I.~Piatetski-Shapiro, 
		On the distribution of prime numbers in the sequence of the form $\fl{f(n)}$,
		\emph{Mat.\ Sb.} 33 (1953), 559--566.
			
		\bibitem{RiSa}
		J.~Rivat and S.~Sargos,
		Nombres premiers de la forme $\fl{n^c}$.
		\emph{Canad.\ J.\ Math.} 53 (2001), no.~2, 414--433.
		
		\bibitem{RiWu}
		J.~Rivat and J.~Wu,
		Prime numbers of the form $\fl{n^c}$.
		\emph{Glasg.\ Math.\ J.} 43 (2001), no.~2, 237--254. 
		
		\bibitem{RS2006}
		O. Robert and P. Sargos, 
		Three-dimensional exponential sums with monomials. 
		\emph{J. Reine Angew. Math.} 591 (2006), 1--20.
		
		\bibitem{SDP}
		Y. Sun, S. Du and H. Pan,
		Vinogradov's theorem with Piatetski-Shapiro primes.
		\emph{Int. Math. Res. Not. IMRN} 2025, no. 15, rnaf125.
		
		\bibitem{Vaal}
		J.~D.~Vaaler,
		Some extremal problems in Fourier analysis. 
		\emph{Bull.\ Amer.\ Math.\ Soc.} 12 (1985), 183--216.
		


		
	\end{thebibliography}
\end{document}